\documentclass[final]{siamart190516}
\usepackage{amsmath}
\usepackage{amsfonts}
\usepackage{subfigure}
\usepackage[algo2e]{algorithm2e}
\usepackage{algorithmic}
\usepackage{array}
\usepackage{url}
\usepackage{hyperref}
\usepackage{url}
\usepackage{bm}
\usepackage{graphicx}
\usepackage{verbatim}
\usepackage{placeins}
\usepackage{multicol}
\usepackage{capt-of}
\usepackage{booktabs}

\DeclareMathOperator{\Tr}{Tr}

\newcommand{\TheTitle}{A template for gradient norm minimization}
\newcommand{\TheAuthors}{Mihai I. Florea}

\newcommand{\defq}{\overset{\operatorname{def}}{=}}
\newcommand{\prox}{\operatorname{prox}}
\newcommand{\inn}[2]{= \overline{#1,#2}}
\newcommand{\on}[1]{\operatorname{#1}}

\newcommand{\Umax}{U_{\on{max}}}

\newcommand{\Lmax}{L_{\on{max}}}

\title{\TheTitle}

\author{Mihai I. Florea\thanks{Department of Mathematical Engineering (INMA), Catholic University of Louvain (UCL), Belgium. E-mail: \email{mihai.florea@uclouvain.be}}.}

\headers{OCGM-G}{\TheAuthors}

\begin{document}

\maketitle

\begin{abstract}
The gradient mapping norm is a strong and easily verifiable stopping criterion for first-order methods on composite problems. When the objective exhibits the quadratic growth property, the gradient mapping norm minimization problem can be solved by online parameter-free and adaptive first-order schemes with near-optimal worst-case rates. In this work we address problems where quadratic growth is absent, a class for which no methods with all the aforementioned properties are known to exist. We formulate a template whose instantiation recovers the existing Performance Estimation derived approaches. Our framework provides a simple human-readable interpretation along with runtime convergence rates for these algorithms. Moreover, our template can be used to construct a quasi-online parameter-free method applicable to the entire class of composite problems while retaining the optimal worst-case rates with the best known proportionality constant. The analysis also allows for adaptivity. Preliminary simulation results suggest that our scheme is highly competitive in practice with the existing approaches, either obtained via Performance Estimation or employing Accumulative Regularization.
\end{abstract}

\section{Introduction}

A very wide range of applications in mathematics, statistics, machine learning and signal processing can be modeled as large-scale composite problems. The size of the problems precludes the computation of second-order information and optimization algorithms may only employ objective function values and gradient-like quantities such as the composite gradient mapping. Two properties render this class tractable: the convexity of the objective, allowing us to find a global optimum, and a local Lipschitz property of the gradient-like quantities enabling acceleration.

Limitations in computational resources such as expended energy or total running time make it desirable to have an accurate quality measure available at runtime for the current estimate of the solution $\hat{x}_k$ at iteration $k$ produced by the optimization scheme. Given an objective function $F(x)$, the gradient-like quantity $g(x)$, the distance $d(x)$ from $x$ to the optimum set $X^*$ and the optimal value $F^*$, the three most commonly used evaluation criteria for $\hat{x}_k$ are the distance to optimum $d(\hat{x}_k)$, the function residual $F(\hat{x}_k) - F^*$ and the norm of the gradient-like quantity $\|g(\hat{x}_k)\|_*$.

Let us examine each metric in turn. In many applications, such as signal restoration tasks, the most desirable criterion is $d(\hat{x}_k)$. However, unless the problem has additional structure, it has been shown in \cite{ref_019} that there exists a worst-case objective function for which no rate on $d(\hat{x}_k)$ can be obtained by \emph{any} optimization method employed. On the other hand, in many loss minimization applications found, for instance, in economics and industrial engineering, the obvious quality measure is the function residual. Whereas the worst-case rate of $\mathcal{O}(1/k^2)$~\cite{ref_019} can be obtained by accelerated schemes such as the Fast Gradient Method~\cite{ref_018} and its generalizations (see, e.g. \cite{ref_020,ref_002,ref_023}), the actual evaluation of this criterion requires the knowledge of $F^*$, which is seldom available.

The remaining alternative is the norm $\|g(\hat{x}_k)\|_*$, being the only criterion among the three that is readily computable at runtime. The optimal rate of $\mathcal{O}(1/k^2)$~\cite{ref_016,ref_017} for minimizing the composite gradient norm on the composite problem class implies the optimal rate for the function residual, rendering the $\|g(\hat{x}_k)\|_*$ criterion the stronger of the two. It can also be very useful on its own. For instance, the gradient norm pertaining to the dual formulation of equality constrained strongly convex minimization measures both the feasibility and the optimality of the primal solution~\cite{ref_021}.

In the context of composite problems, the behavior of optimization schemes is strongly effected by the existence of the quadratic functional growth (QFG)~\cite{ref_015} property, which can be considered a watershed within this class. When the objective exhibits QFG with parameter $\mu > 0$, the three criteria can be minimized optimally simultaneously, that is by the same scheme at the same point, and the worst-case rate for all of them is linear and given by $\mathcal{O}\left(\exp\left(-c \sqrt{\frac{\mu}{L}} k\right)\right)$ with $L$ being the largest local Lipschitz constant encountered and $c$ being an algorithm and criterion dependent constant. The Accelerated Gradient Method with Memory with adaptive restart (R-AGMM) proposed in \cite{ref_007} has virtually all the desirable characteristics needed to address this class: its possesses the aforementioned optimal worst-case rate, it is parameter-free, meaning that it requires no knowledge of the problem parameters, it is adaptive because it dynamically estimates $L$ at every iteration, being further equipped with a runtime convergence guarantee adjustment procedure, and lastly it is an online method that progressively produces better estimates of the solution without the need to specify in advance the number of iterations or an accuracy threshold. With such a collection of properties, R-AGMM essentially solves the problem of composite minimization with quadratic growth.

When quadratic growth is absent, the distance to optimum criterion does not allow a worst-case rate whereas the problem of function residual minimization is solved by the Accelerated Multistep Gradient Scheme (AMGS)~\cite{ref_020}, the Accelerated Composite Gradient Method (ACGM)~\cite{ref_009,ref_010} as well as AGMM~\cite{ref_006,ref_007}, because all aforementioned schemes are optimal, parameter-free, adaptive and online. The challenge remains to develop a method with the four properties that minimizes the gradient mapping norm, a stronger condition.

One of the first works to address this criterion \cite{ref_021} artificially imposes the watershed property: it alters the objective by adding a quadratic regularizer. The regularization parameter depends on a target accuracy of the composite gradient norm rendering this approach offline only. By solving the regularized problem with a version of FGM designed to take advantage of strong convexity, this technique attains a worst-case rate that differs from the optimal one by a logarithmic term in the target accuracy that is inherent in the analysis and cannot be removed.

The first truly optimal approach was proposed in \cite{ref_011} in the form of the Optimized Gradient Method for minimizing the Gradient norm (OGM-G). The derivation of this method was based on the numerical output of the Performance Estimation Problem (PEP)~\cite{ref_004} under the criterion of minimizing the gradient norm for a given initial function residual. A related method is the Optimized Gradient Method~\cite{ref_012}, developed in the same manner but under the different criterion of optimizing the function residual subject to an initial distance to the optimum. Running OGM for a certain $T$ number of iterations, followed by OGM-G for the same number $T$ iterations produces a combination scheme that is optimal and currently has the best known rate on the subproblem of smooth unconstrained minimization~\cite{ref_022}. A quasi-online behavior can be obtained by branching off an instance of OGM-G as a subprocess of OGM every $T$ iterations, with $T$ increasing exponentially. However, even this combination scheme cannot deal with composite problems and needs to be supplied the value of Lipschitz constant $L$ to converge. Also, the opaque nature of the PEP does not provide a theoretical tool for analyzing OGM-G that would facilitate the extension to composite problems or the construction of an adaptive mechanism.

A preliminary study in \cite{ref_003} attempted to derive a Lyapunov (potential) function mathematical tool for achieving the above goals, without success. Interestingly, the method proposed in \cite{ref_003}, which we denote by OGM-G-DW, actually differs from OGM-G and has a worst-case rate constant lower by a factor of $4$.

Another attempt in \cite{ref_014}, seemingly based on Performance Estimation, has managed to produce a valid Lyapunov function for OGM-G. However, this function does not appear to have a clear explanation. To remedy this shortcoming, the work in \cite{ref_014} seeks to find a geometric structure in optimal first-order methods, independent of the Lyapunov analysis. It observes that in FGM the oracle points are obtained from the previous two main iterates by extrapolation. The analysis speculates that since FGM optimally minimizes the function residual, the extrapolation property should be found in a method that optimally minimizes the composite gradient norm as well. By enforcing this assumption in the computer assisted analysis of OGM-G, the authors of \cite{ref_014} formulate FISTA-G, a method optimal under the same criteria as OGM-G, that is further applicable to composite problems. However, FISTA-G also inherits the weaknesses of OGM-G: it relies on $L$ being known and the derivations do not provide a means of adding an adaptive mechanism. Moreover, the computer assisted analysis cannot explain the meaning of the quantities involved in the algorithm updates and in the convergence proof.

A very recent work in \cite{ref_013} addresses the gradient norm minimizationn problem without resorting to Performance Estimation. Instead, it employs Accumulative Regularization (AR), a meta-algorithm that combines the restarting and regularization techniques in \cite{ref_015} and \cite{ref_021}, respectively, to attain an optimal and parameter-free scheme. Successive runs of FGM or a similar accelerated scheme are performed on quadratically regularized objectives for exponentially \emph{decreasing} numbers of iterations, each time exponentially \emph{increasing} the regularization parameter. The regularizer centers are weighted averages of the previous run outputs, the weights being the corresponding regularization parameters. Interestingly, this approach attains an optimal rate without supplying the quadratic growth parameter of the auxiliary problem objectives to the accelerated schemes used to solve them, a technique that resembles the adaptive restart proposed in \cite{ref_001}. This approach is very transparent but the exponentially increasing perturbation of the objective does not leave room for an online mode of operation.

In this work we provide a simple template for minimizing the norm of a gradient-like quantity given an initial function residual. Instantiating this template we recover both OGM-G and FISTA-G, providing a clear explanation for the algorithmic quantities associated with these methods. The template also reveals the quantities these methods are actually minimizing at runtime. Moreover, using the template we can enhance FISTA-G to obtain an Optimized Composite Gradient Method for minimizing the Gradient norm (OCGM-G), and optimal method possessing the currently highest known worst-case rate. Combining ACGM with OCGM-G produces a quasi-online parameter-free method that is fully-adaptive in the ACGM part and allows for adaptivity in the OCGM-G part. Preliminary simulation results show that our approach can outperform the state-of-the-art.

\section{Problem formulations and bounds}

Large-scale composite problems are the main focus of this work but we also need to study the simpler class of smooth unconstrained minimization. Performance Estimation~\cite{ref_004} can only reliably produce numerical instances of optimal methods on this narrower class~\cite{ref_024} and OGM-G~\cite{ref_011} was designed to function exclusively under these restrictive assumptions. Several lessons learned from OGM-G will also apply to composite minimization, as we will see in the sequel. At this stage, we provide the gradient-like quantities, the gradient steps as well as the usable bounds on both classes. All results in this section for composite problems are generalizations of those for smooth unconstrained problems with the exception of the bounds, which are tighter for the smaller class enabling faster convergence.

\subsection{Smooth unconstrained minimization} We first consider the optimization problem
\begin{equation} \label{label_001}
\min_{x \in \mathbb{E}} F(x) = f(x),
\end{equation}
having an objective function $f:\mathbb{E} \rightarrow \mathbb{R}$ that is convex and differentiable on the entire domain $\mathbb{E} \defq \mathbb{R}^n$. The gradient $g = f'$ is $L > 0$ Lipschitz meaning that $f$ is $L$-smooth. Throughout this work we use the Euclidean norm as $\| x \| \defq \sqrt{ \langle B x, x \rangle }$, $x \in \mathbb{E}$, with $B \succ 0$ with the dual norm given by $\| s \|_* \defq \sqrt{ \langle s, B^{-1} s \rangle }$, $s \in \mathbb{E}^*$.

We define the gradient step operator as $T_{\hat{L}}(z) = z - \frac{1}{\hat{L}} B^{-1} f'(z)$ for any estimate $\hat{L} > 0$ of $L$. A bound that completely describes the problem~\cite{ref_024} can be written~\cite{ref_003,ref_008} using the gradient step operator as
\begin{equation} \label{label_002}
f(x) \geq f(y) + \langle f'(y), T_L(x) - y \rangle + \frac{1}{2 L} \| f'(x) \|_*^2 + \frac{1}{2 L} \| f'(y) \|_*^2, \quad x, y \in \mathbb{E}.
\end{equation}

\subsection{Composite minimization}

We next turn our attention to the broader class of composite minimization, formally described as
\begin{equation} \label{label_003}
\min_{x \in \mathbb{E}} F(x) = f(x) + \Psi(x),
\end{equation}
with the properties of $f$ carrying over from the previous problem. We further assume that $L$ may not be known and that the optimization scheme starts with an initial arbitrary guess $L_0 > 0$. The regularizer $\Psi$ is an extended value convex lower-semicontinuous function for which the exact computation of the proximal operator, given by
\begin{equation} \label{label_004}
\prox_{\tau}(x) \defq \arg\min_{z \in \mathbb{E}}\left\{\tau \Psi(z) + \frac{1}{2} \| z - x \|^2 \right\}, \quad x \in \mathbb{E},
\end{equation}
is tractable for all $\tau > 0$. Constraints are imposed by setting the regularizer to be infinite outside the feasible set $Q$. The gradient-like quantity for this class is given by the gradient mapping~\cite{ref_020}
\begin{equation} \label{label_005}
g_{\hat{L}}(x) \defq \hat{L} B^{-1} (x - T_{\hat{L}}(x)), \quad x \in \mathbb{E},
\end{equation}
where $\hat{L} > 0$ and $T_{\hat{L}}$ is the proximal step operator, defined as
\begin{equation} \label{label_006}
T_{\hat{L}}(x) \defq \arg\min_{z \in \mathbb{E}} \left\{ f(x) + \langle f'(x), z - x \rangle + \frac{\hat{L}}{2} \| z - x \|^2 + \Psi(z) \right\}.
\end{equation}
On this class we use the following bounds~\cite{ref_008} that allow for adaptivity.
\begin{proposition} \label{label_007}
For any $y \in \mathbb{E}$ and $\hat{L} > 0$ satisfying the descent condition
\begin{equation} \label{label_008}
f(T_{\hat{L}}(y)) \leq f(y) + \langle f'(y), T_{\hat{L}}(y) - y \rangle + \frac{\hat{L}}{2} \| T_{\hat{L}}(y) - y \|^2,
\end{equation}
we can construct a global lower bound on the objective of the form
\begin{equation} \label{label_009}
F(x) \geq F(T_{\hat{L}}(y)) + \frac{1}{2 \hat{L}} \| g_{\hat{L}}(y) \|_*^2 + \langle g_{\hat{L}}(y), x - y \rangle, \quad x \in \mathbb{E}.
\end{equation}
\end{proposition}
For both problems we assume the optimal set $X^*$ of points $x^*$ is non-empty.

\section{Revisiting quadratic growth} \label{label_010}

Let us first show that the problem of composite gradient norm minimization when the objective \emph{does} exhibit quadratic growth is essentially solved. For conciseness, we omit proofs in this section. For an extensive presentation with complete proofs we refer the reader to \cite[Section 3.3]{ref_007}.

\subsection{Definition and properties}
For a growth parameter $\mu \geq 0$, the Quadratic Functional Growth (QFG) property is stated as
\begin{equation} \label{label_011}
F({x}) \geq F^* + \frac{\mu}{2} d({x}_0)^2, \quad {x} \in \mathbb{E} ,
\end{equation}
where $\pi({x})$ and $d({x})$ are the projection onto and the distance to the optimal set $X^*$, respectively defined as
\begin{equation}
\pi({x}) \defq \displaystyle \arg\min_{{x^*} \in X^*}\|{x^*} - {x}\|, \quad d({x}) \defq \| {x} - \pi({x}) \|, \quad {x} \in \mathbb{E}.
\end{equation}
A link between the subgradient norm and the function residual is provided by a variation of the Polyak-\L{}ojasiewicz inequality with the tight constant $\frac{2}{\mu}$, stated as
\begin{equation}
F(x) - F^* \leq \frac{2}{\mu} \| \xi \|_*^2, \quad x \in \mathbb{E}, \quad \xi \in \partial F(x).
\end{equation}

\subsection{A simple method with all the key strengths}
We describe and analyze in the following ACGM with adaptive restart (R-ACGM), one of the simplest optimal online adaptive parameter-free methods addressing this class. R-ACGM is a two-level scheme wherein the wrapper restart meta-algorithm (Algorithm~\ref{label_012}) repeatedly calls instances of ACGM (Algorithm~\ref{label_014}).

\begin{algorithm}[h]
\footnotesize
\caption{A parameter-free adaptive restart scheme calling ACGM}
\label{label_012}
\begin{algorithmic}[1]
\STATE \textbf{Input:} $r_0 \in \mathbb{R}^n$, $0 < \sigma < \frac{1}{2}$, $s > 1$, $L_0 > 0$, $\gamma_d \in (0, 1]$, $\gamma_u > 1$
\STATE $(r_1, U_1, n_1) = \on{ACGM}(r_0, +\infty, E_0, L_0, \gamma_d, \gamma_u)$ \hspace{1em} \#Reference step \label{label_013}
\STATE $\bar{U}_1 := U_1$
\FOR{$j = 1, 2, \ldots{}$}
\STATE $(r_{j + 1}, U_{j + 1}, n_{j + 1}) = \on{ACGM}(r_j, \bar{U}_j, E_j, L_0, \gamma_d, \gamma_u)$
\IF {$E_j$ holds}
\STATE \textbf{break}
\ENDIF
\IF {$F(r_{j}) - F(r_{j + 1}) \leq \frac{\sigma}{1 - \sigma} (F(r_{j - 1}) - F(r_{j}))$}
\STATE $\bar{U}_{j + 1} := \bar{U}_j$ \hspace{1em} \# Normal step
\ELSE
\STATE $\bar{U}_{j + 1} := s \cdot \bar{U}_j$ \hspace{1em} \# Adjustment step
\ENDIF
\ENDFOR
\end{algorithmic}
\end{algorithm}

\begin{algorithm}[h]
\footnotesize
\caption{$(r_{j + 1}, U_{j + 1}, n_{j + 1}) = \on{ACGM}(r_j, \bar{U}_j, E_j, L_0, \gamma_d, \gamma_u)$}
\label{label_014}
\begin{algorithmic}[1]
\STATE $x_0 = v_0 = r_j$, $A_0 = 0$, $k := 0$
\REPEAT
\STATE $L_{k + 1} := \gamma_d L_k$
\LOOP
\STATE $a_{k + 1} := \frac{1 + \sqrt{1 + 4 L_{k + 1} A_k}}{2 L_{k + 1}}$
\STATE $y_{k + 1} := \frac{1}{A_k + a_{k + 1}}(A_k {x}_k + a_{k + 1} {v}_k)$
\STATE ${x}_{k + 1} := T_{L_{k + 1}}(y_{k + 1})$
\IF {$f({x}_{k + 1}) \leq f(y_{k + 1}) + \langle f'(y_{k + 1}), {x}_{k + 1} - y_{k + 1} \rangle + \frac{L_{k + 1}}{2} \| {x}_{k + 1} - y_{k + 1} \|^2$}
\STATE Break from loop
\ELSE
\STATE $L_{k + 1} := \gamma_u L_{k + 1}$
\ENDIF
\ENDLOOP
\STATE $A_{k + 1} = A_k + a_{k + 1}$
\STATE $v_{k + 1} = v_k + a_{k + 1} L_{k + 1} (x_{k + 1} - y_{k + 1})$
\STATE $k := k + 1$
\UNTIL $A_{k} \geq \bar{U}_j$ or $E_j$ holds
\STATE $r_{j + 1} = \arg\min\{F(x_{k}), F(x_1)\}$
\STATE $U_{j + 1} = A_{k}$
\STATE $n_{j + 1} = k$
\end{algorithmic}
\end{algorithm}

The wrapper scheme takes as input the starting point $r_0$, which needs to be feasible. If uncertain, a starting point can be made feasible by either projecting it onto $Q$ or by calling $\prox_{\tau}(r_0)$ with arbitrary $\tau_0 > 0$. The scalar $\sigma$ can be either the value $\sigma^* \approx 0.0981709$, shown in \cite{ref_007} to be optimal, or the value $\sigma = e^{-2}$, found to work well in practice. The parameter $s$ multiplicatively increases the threshold $\bar{U}_j$ at wrapper iteration $j \geq 1$, if necessary, and experiments show that a good choice is $s = 4$. This value is also optimal when dealing with a resisting oracle. The initial $L_0$ can be set to $1$ unless a better guess of $L$ is available. The line-search parameters $\gamma_d$ and $\gamma_u$ can be set to $0.9$ and $2.0$, respectively, as per original ACGM specification~\cite{ref_009}. We denote by $L(r_j)$ the Lipschitz estimate $\hat{L}$ returned by the line-search procedure of ACGM for each evaluated point $y = r_j$ satisfying \eqref{label_008} in Proposition~\ref{label_007}.

The heuristic $E_0$ should ensure that $U_1 = \bar{U}_1 < s \Umax$ with $\Umax \defq \frac{1}{\mu \sigma}$. One such criterion is described in \cite{ref_001} and returns $U_1 = A_k$ with $k$ being the first ACGM iteration satisfying
\begin{equation} \label{label_015}
F(x_m) - F(x_k) \leq \frac{\sigma}{1 - \sigma} ( F(x_0) - F(x_m) ),\quad m = \left\lceil \frac{k}{\sqrt{s}} \right\rceil \geq 1.
\end{equation}

The termination criterion $E_j$, $j \geq 1$, can be triggered by the spent computational budget, which can be the total number of ACGM iterations, total running time, total expended energy or any other similar constraint.

\subsection{Convergence analysis}

The worst-case rate of ACGM in Algorithm~\ref{label_014} is, according to \cite[Appendix~E]{ref_009}, given by
\begin{equation} \label{label_016}
F({x}_k) - F^* \leq \frac{1}{2 A_k} d({x}_0)^2, \quad k \geq 1,
\end{equation}
with $A_k$ lower bounded as $A_k \geq \frac{(k + 1)^2}{4 L_u}$, where $L_u \defq \max\{\gamma_d L_0, \gamma_u L\}$ is the worst-case Lipschitz estimate. We can thus approximate $A_k$ from below using the monomial term $c k^p$ with $c = \frac{1}{4 L_u}$ and $p = 2$.

The heuristic $E_0$ ensures that $k \leq \lceil \sqrt{4 L_u s \Umax}\rceil$ (see \cite[Lemma 3.4]{ref_007}) and the worst-case overall rate for R-ACGM in ACGM iterations $N$ is given by
\begin{equation} \label{label_017}
\begin{gathered}
d(x_N) \leq \sqrt{\frac{2}{\mu} (F({x}_N) - F^*)}, \quad
F({x}_N) - F^* \leq \frac{2}{\mu} \left( 1 + \frac{L_f}{L({r}_j)} \right)^2 \| g_{L({r}_j)}({r}_j) \|_*^2, \\
\| g_{L({r}_j)}({r}_j) \|_*^2 \leq 2L({r_{j}}) C_b
\exp{\left(-\eta(\sigma) \frac{1}{e} \sqrt{\frac{\mu}{s L_u}} (N + N_o)\right)} \left( F({r}_0) - F({r}_1) \right),
\end{gathered}
\end{equation}
where $j$ marks the last restart before $N$ iterations have elapsed, the adjustment overhead is given by $C_b \defq \left(\frac{1 - \sigma}{\sigma} \right)^{b + 1} \prod_{i \in \mathbb{B}} \frac{F(r_i) - F(r_{i + 1})}{F(r_{i - 1}) - F(r_i)}$, with the total number of adjustments $b \leq \left\lceil log_s\left(\frac{U_{\operatorname{max}}}{U_1}\right)\right\rceil$ and the iteration efficiency $\eta(\sigma) \in (0,1)$, defined as the efficiency with respect to the case when $\mu$ is known, has the expression
\begin{equation} \label{label_018}
\eta(\sigma) \defq -\frac{e}{2} \sqrt{\sigma} \ln\left(\frac{\sigma}{1 - \sigma}\right) , \quad \sigma \in \left(0, \frac{1}{2} \right).
\end{equation}
The total overhead gain in iterations satisfies $N_o \defq (b + 1) \tilde{n} - (N_r + N_b + N_t) \geq 0$, with the maximal number of iterations between successive restarts given by $\tilde{n} = \left\lfloor 2 \sqrt{ \frac{s L_u}{\mu \sigma} }\right\rfloor$. The number of iterations pertaining to the first instance of ACGM, the adjustments as well as those elapsed from the last restart are, respectively, given by $N_r \leq \tilde{n} {s}^{-\frac{b}{2}}$, $N_b \leq \tilde{n} \cdot{} \min\left\{b, \frac{\sqrt{s}}{\sqrt{s} - 1}\right\}$ and $N_t \leq \tilde{n}$. The value $\sigma^*$ maximizes the efficiency with $\eta(\sigma^*) \approx 0.9444$ whereas $\sigma = e^{-2}$ gives $\eta(e^{-2}) \approx 0.9273$, both efficiency levels being very close to $1$.

We thus see in \eqref{label_017} that all three convergence criteria are minimized optimally at the same point by the online parameter-free adaptive scheme R-ACGM described in Algorithms \ref{label_012} and \ref{label_014}.

\section{A template for gradient norm minimization given an initial function residual}

We now turn our attention to the more difficult scenario where quadratic growth is too small to be effectively harnessed and analyze the common aspects of smooth unconstrained and composite minimization.

We denote the points where the oracle is called by $y_k$, $k \geq 1$. We consider that for each oracle point, the algorithm obtains a Lipschitz estimate $L_k$, $k \geq 1$, with $L_0$ being an initial guess for $L$, as well as the gradient-like quantity
\begin{equation} \label{label_019}
g_k \defq g_{L_k}(y_k), \quad k \geq 1.
\end{equation}
We have $y_1 = x_0$, where $x_0$ is the initial point, and the gradient step points $x_k$, $k \geq 1$, are respectively given by
\begin{equation} \label{label_020}
x_{k} \defq T_{L_{k}}(y_{k}) \overset{\eqref{label_005}}{=} y_k - \frac{1}{L_k} B^{-1} g_{L_k}(y_k) \overset{\eqref{label_019}}{=} y_k - \frac{1}{L_k} B^{-1} g_k, \quad k \geq 1.
\end{equation}
Whenever we refer to the initial point $x_0$ and the gradient step points $x_k$ with $k \geq 1$ collectively, we use the term x-points.

\subsection{The algorithmic template}

We propose the template outlined in Algorithm~\ref{label_021} that minimizes the gradient mapping norm subject to an initial function residual. It takes as input the starting point $x_0$ and the number of iterations $T$. The template itself is not online and $T \geq 2$ must be specified in advance. For this reason, throughout our analysis, we will use bounded integer ranges where $k \inn{i_1}{i_2}$ denotes $k \in [i_1, i_2] \cap \mathbb{Z}$.

\begin{algorithm}[h]
\caption{A template for minimizing $\|g_T\|$ given $F(x_0) - F^*$}
\label{label_021}
\begin{algorithmic}[1]
\STATE \textbf{Input:} $x_0 \in \mathbb{E}$, $T \geq 2$
\STATE \textbf{Output:} $y_T$, $g_T$, $L_T$, $x_T$
\STATE $s_0 = 0$ \label{label_021_s0}
\FOR{$k = 0, \ldots{}, T - 1$}
\STATE $y_{k + 1} = \left\{\begin{array}{ll}x_0, & k = 0 \\ x_k - b_k B^{-1} s_k, & k \geq 1 \end{array}\right.$ \label{label_021_y}
\STATE Compute $L_{k + 1}$ based on $y_{k + 1}$ \label{label_021_var}
\STATE $x_{k + 1} = T_{L_{k + 1}}(y_{k + 1})$
\STATE $g_{k + 1} = L_{k + 1} B (y_{k + 1} - x_{k + 1})$ \label{label_021_L_new}
\IF {$k = T - 1$}
\STATE break
\ENDIF
\STATE Compute $a_{k + 1}$ and $b_{k + 1}$
\STATE $s_{k + 1} = s_k + a_{k + 1} g_{k + 1}$ \label{label_021_s_new}
\ENDFOR
\end{algorithmic}
\end{algorithm}

To instantiate the template into an actual algorithm, one only needs to define a means of obtaining the Lipschitz estimates $L_k$ for each $k \inn{1}{T}$ and, very importantly, determine update rules for the \emph{positive} sequences $\{a_k\}_{k \inn{1}{T - 1}}$, $\{b_k\}_{k \inn{1}{T-1}}$, as shown in line~\ref{label_021_var}.

To aid us in our analysis, we define the forward accumulators $A_k \defq A_0 + \sum_{i = 1}^{k} a_i$, $k \inn{1}{T - 1}$ with $A_0 > 0$ and the ranged accumulators $B_{k,\bar{k}} \defq \sum_{i = k}^{\bar{k} - 1} b_i$ where both $k$ and $\bar{k}$ span $\overline{1,T}$, noting that $B_{k, \bar{k}} = 0$ when $k \geq \bar{k}$. We thus have for all $T \geq 2$ that
\begin{gather}
A_{k} + a_{k + 1} = A_{k + 1}, \quad a_{k + 1} > 0, \quad A_0 > 0, \quad k \inn{0}{T-2}, \label{label_027} \\
B_{k,T} = B_{k + 1,T} + b_{k}, \quad b_{k} > 0, \quad B_{T,T} = 0, \quad k \inn{1}{T-1}. \label{label_028}
\end{gather}

The template outputs the state parameters with desirable properties obtained after performing $T$ iterations: $y_T$ and $g_T$ as well as $L_T$ and $x_T$, if necessary. Note that the output oracle point $y_T$ requires only $T - 1$ gradient-like quantity computations, excluding the overhead pertaining to line-search or similar mechanisms. This explains why the weight sequences $a_k$ and $b_k$ are only defined up to index $T-1$. As we shall see, when instantiating the template for OGM-G, the value of $a_T$ becomes infinite.

For the sake of completeness let us study what happens if we use $T = 1$ in Algorithm~\ref{label_021}. Our template reduces to one step of the Gradient Method with the sequences $a_1$ and $b_1$ having no impact on the behavior of the scheme. We instead have $y_1 = x_0$ and $x_1 = T_{L_1}(x_0)$ and the well-known descent rule~\cite{ref_019,ref_020} implies
\begin{equation} \label{label_029}
\|g_1\|_*^2 \leq 2 L_1 (F(x_0) - F(x_1)) \leq 2 L_1 (F(x_0) - F^*).
\end{equation}

\subsection{Equivalent forms}

Gradient norm minimization algorithms have been proposed using a variety of update rules. For instance, the calculation of the oracle point in the original OGM-G~\cite{ref_011} relies on the previous two x-points and the previous oracle point. This formulation was been denominated in \cite{ref_014} as the momentum form and extended to a variety of optimal first-order schemes. A subexpression in the momentum oracle point update can be updated separately, resulting in an auxiliary iterate form. This refactoring allows for an infinite number of auxiliary iterate expressions. The one listed for OGM-G in \cite{ref_011} was chosen to resemble the OGM counterpart but \cite{ref_014} proposes the only auxiliary sequence that can be updated independently in an \emph{additive} fashion. Due to this an several other useful properties we consider the sequence in \cite{ref_014} to be the canonical reformulation. Also, \cite{ref_003} lists a variant of OGM-G with poorer convergence guarantees employing \emph{two} separate auxiliary iterates. To better compare with existing methods and to discover new insights into the process of optimal gradient norm minimization, we investigate how our template relates to these other formulations by rearranging it accordingly.

In the momentum form, the oracle points $y_{k + 1}$, $k \geq 2$, are obtained from the previous two x-points $x_k$ and $x_{k - 1}$, using extrapolation with a factor of $\mathcal{E}_k > 0$ while adding a gradient step correction term with parameter $\mathcal{C}_k > 0$, according to
\begin{equation}
y_{k + 1} = x_k + \mathcal{E}_k (x_k - x_{k - 1}) + \mathcal{C}_k (x_k - y_k), \quad k \geq 2.
\end{equation}
Note that we can restrict our analysis to $k \geq 2$, if applicable, because the first two oracle points have simple expressions that do not utilize weight accumulation, namely $y_1 = x_0$ and $y_2 = x_1 - b_1 a_1 g_1$.
From Algorithm~\ref{label_021} we infer that
\begin{equation} \label{label_030}
\begin{gathered}
x_k - x_{k - 1} = y_k - \frac{1}{L_k} B^{-1} g_k - x_{k - 1} = -b_{k - 1} B^{-1} s_{k - 1} - \frac{1}{L} B^{-1} g_k \\
\overset{\on{line}\ref{label_021_s_new}}{=} -b_{k - 1} B^{-1} (s_k - a_k g_k) - \frac{1}{L_k} B^{-1} g_k =
-b_{k - 1} B^{-1} s_k - \left(\frac{1}{L_k} - a_k
b_{k - 1}\right) B^{-1} g_k \\ \overset{\on{line}\ref{label_021_L_new}}{=} -b_{k - 1} B^{-1} s_k - \left(1 - L_k a_k b_{k - 1}\right) (y_k - x_k), \quad k \geq 2.
\end{gathered}
\end{equation}
From line~\ref{label_021_y} we also have that
\begin{equation} \label{label_031}
-b_{k - 1} B^{-1} s_k = \frac{b_{k - 1}}{b_k}(y_{k + 1} - x_k), \quad k \geq 2.
\end{equation}
Substituting \eqref{label_031} in \eqref{label_030} and refactoring gives
\begin{equation} \label{label_032}
y_{k + 1} = x_k + \frac{b_k}{b_{k - 1}}(x_k - x_{k - 1}) + b_k \left(L_k a_k - \frac{1}{b_{k - 1}} \right) (x_k - y_k), \quad k \geq 2,
\end{equation}
and hence
\begin{equation} \label{label_033}
\mathcal{E}_k = \frac{b_k}{b_{k - 1}}, \quad \mathcal{C}_k = b_k \left(L_k a_k - \frac{1}{b_{k - 1}} \right) , \quad k \geq 2.
\end{equation}
We list the extrapolation based equivalent template in Algorithm~\ref{label_034}.

\begin{algorithm}[h]
\caption{Extrapolation based form of Algorithm~\ref{label_021}}
\label{label_034}
\begin{algorithmic}[1]
\STATE \textbf{Input:} $x_0 \in \mathbb{E}$, $T \geq 2$
\STATE \textbf{Output:} $y_T$, $g_T$, $L_T$, $x_T$
\FOR{$k = 0, \ldots{}, T - 1$}
\STATE $y_{k + 1} = \left\{\begin{array}{ll}x_0, &k = 0 \\
x_1 - b_1 a_1 B^{-1} g_1, & k = 1 \\
x_k + \frac{b_k}{b_{k - 1}}(x_k - x_{k - 1}) + b_k \left(L_k a_k - \frac{1}{b_{k - 1}} \right) (y_k - x_k), & k \geq 2 \end{array}\right.$\label{label_035}
\STATE Based on $y_{k + 1}$, compute $L_{k + 1}$, $a_{k + 1}$ and $b_{k + 1}$ \label{label_036}
\STATE $x_{k + 1} = T_{L_{k + 1}}(y_{k + 1})$
\STATE $g_{k + 1} = L_{k + 1} B (y_{k + 1} - x_{k + 1})$ \label{label_037}
\ENDFOR
\end{algorithmic}
\end{algorithm}

We next turn our attention to the efficient form described in \cite{ref_011}, whereby the next oracle point is obtained via a convex combination between the previous x-point and an auxiliary sequence, denoted by $(v_k)_{k \inn{0}{T - 1}}$, that is updated by adding a constant times the new gradient-like quantity at each iteration. Note that $(v_k)_{k \inn{0}{T - 1}}$ differs markedly from our gradient accumulation sequence $(s_k)_{k \inn{0}{T - 1}}$ in Algorithm~\ref{label_021}.

For every $k \geq 0$, we denote the scalar step sequence $(c_k)_{k \inn{1}{T - 1}}$ with positive entries defining the auxiliary sequence update and the multiplier sequence $(d_k)_{k \inn{1}{T - 1}}$ with all entries in $[0,1]$ blending $x_k$ and $v_k$ in the oracle point update. We have
\begin{align}
y_{k + 1} &= (1 - d_{k}) x_k + d_{k} v_k, \quad k \inn{1}{T-1}, \label{label_038} \\
v_{k + 1} &= v_k - c_{k + 1} B^{-1} g_{k + 1}, \quad k \inn{0}{T-2}, \label{label_039}
\end{align}
where $v_0 = x_0$. The first update in \eqref{label_038} implies
\begin{equation} \label{label_040}
y_{k + 1} = x_k + d_{k} (v_k - x_k), \quad k \inn{1}{T-1}.
\end{equation}
We can unroll the calculation of $v_k$ and $x_k$ up to $x_0$ for $k \inn{1}{T-1}$ as follows:
\begin{gather}
v_{k} \overset{\eqref{label_039}}{=} x_0 - \sum_{i = 1}^{k} c_i B^{-1} g_i, \label{label_041}\\
x_{k} \overset{\eqref{label_020}}{=} y_k - \frac{1}{L_k} B^{-1} g_k \overset{\on{line}\ref{label_021_y}}{=} x_{k - 1} - b_{k - 1} B^{-1} s_{k - 1} - \frac{1}{L_k} B^{-1} g_k \label{label_042}
\\ \overset{\eqref{label_020}}{=} y_{k - 1} - \frac{1}{L_{k - 1}} B^{-1} g_{k - 1} - b_{k - 1} B^{-1} s_{k - 1} - \frac{1}{L_k} B^{-1} g_k = ... = x_0 - \mathcal{L}_k - \mathcal{S}_{k - 1}, \nonumber
\end{gather}
where
\begin{equation} \label{label_043}
\mathcal{L}_k \defq \sum_{i = 1}^{k} \frac{1}{L_i} B^{-1} g_i, \quad \mathcal{S}_k \defq \sum_{i = 1}^{k} b_i B^{-1} s_i, \quad k \inn{0}{T - 1}.
\end{equation}
From lines \ref{label_021_s0} and \ref{label_021_s_new} in Algorithm~\ref{label_021} we can also unroll the computation of $s_k$ as
\begin{equation} \label{label_044}
s_k = \sum_{i = 1}^{k} a_i g_i, \quad k \inn{0}{T - 1}.
\end{equation}
We can refactor $\mathcal{S}_k$ using the ranged accumulator notation, taking into account that $a_{k + 1} B_{k + 1, k + 1} = 0$, for all $k \inn{0}{T - 2}$ as follows:
\begin{equation} \label{label_045}
\mathcal{S}_k = \sum_{i = 1}^{k} b_i B^{-1} \sum_{j = 1}^{i} a_j g_j = \sum_{i = 1}^{k} a_i B_{i, k + 1} B^{-1} g_i = \sum_{i = 1}^{k + 1} a_i B_{i, k + 1} B^{-1} g_i.
\end{equation}
Substituting \eqref{label_041} and \eqref{label_042} using \eqref{label_045} in \eqref{label_040} and rearranging terms we obtain
\begin{equation}
y_{k + 1} = x_k + d_{k} \sum_{i = 1}^{k} \left(\frac{1}{L_i} + a_i B_{i, k} - c_i \right) B^{-1} g_i, \quad k \inn{1}{T-1}.
\end{equation}
For the update in line~\ref{label_021_y} of Algorithm~\ref{label_021} to match this auxiliary iterate form, we must have
\begin{equation}
d_{k} \left(\frac{1}{L_i} + a_i B_{i, k} - c_i \right) = -b_k a_i, \quad i \inn{1}{k}, \quad k \inn{1}{T-1}.
\end{equation}
We define the sequence $\tilde{c}_k = \frac{1}{a_k} \left( c_k - \frac{1}{L_k} \right)$ satisfying $c_k = \frac{1}{L_k} + a_k \tilde{c_k}$, $k \geq 1$ and we thus need to have
\begin{equation}
d_{k} (B_{i,k} - \tilde{c}_i) = -b_k, \quad i \inn{1}{k}, \quad k \inn{1}{T-1}.
\end{equation}
Noting that $B_{i,k} + B_{k, T} = B_{i, T}$ for all $i \inn{1}{k}$ and $k \inn{1}{T-1}$, we arrive at the solution
\begin{equation} \label{label_046}
d_k = \frac{b_k}{B_{k,T}}, \quad \tilde{c}_k = B_{k, T}, \quad c_k = \frac{1}{L_k} + a_k B_{k, T}, \quad k \inn{1}{T-1}.
\end{equation}
Interestingly, taking \eqref{label_046} with $k = T - 1$ implies that $d_{T - 1} = \frac{b_{T - 1}}{B_{T -1, T}} = 1$ and from \eqref{label_038} we have
\begin{equation} \label{label_047}
y_T = v_{T - 1}.
\end{equation}

We can now write down the auxiliary sequence formulation of our template in Algorithm~\ref{label_048}. It produces the same iterates as Algorithm~\ref{label_021}, but relies on the precomputation of the weight sequence $(b_k)_{k \inn{0}{T - 1}}$.

\begin{algorithm}[h]
\caption{Auxiliary sequence form of Algorithm~\ref{label_021}}
\label{label_048}
\begin{algorithmic}[1]
\STATE \textbf{Input:} $x_0 \in \mathbb{E}$, $T \geq 2$
\STATE \textbf{Output:} $y_T = v_{T - 1}$
\STATE Precompute the sequence $(b_k)_{k \inn{1}{T - 1}}$
\STATE $v_0 = x_0$ \label{label_049}
\FOR{$k = 0, \ldots{}, T - 2$}
\STATE $y_{k + 1} = \left\{\begin{array}{ll} x_0, & k = 0 \\ \frac{1}{B_{k, T}} \left(B_{k + 1, T} x_k + b_k v_k \right), & k \geq 1 \end{array} \right.$ \label{label_050}
\STATE Based on $y_{k + 1}$, compute $L_{k + 1}$ and $a_{k + 1}$ \label{label_051}
\STATE $x_{k + 1} = T_{L_{k + 1}}(y_{k + 1})$
\STATE $g_{k + 1} = L_{k + 1} B (y_{k + 1} - x_{k + 1})$ \label{label_052}
\STATE $v_{k + 1} = v_k - \left(\frac{1}{L_{k + 1}} + a_{k + 1} B_{k + 1, T}\right) B^{-1} g_{k + 1}$ \label{label_053}
\ENDFOR
\end{algorithmic}
\end{algorithm}

We can further combine Algorithm~\ref{label_021} with Algorithm~\ref{label_048} to produce a form that relies on the two auxiliary sequences $(s_k)_{k \inn{0}{T - 1}}$ and $(v_k)_{k \inn{0}{T - 1}}$. For instance, the oracle points for OGM-G-DW were updated in \cite{ref_003} using both sequences, an expression that we generalize to
\begin{equation} \label{label_054}
A_{k + 1} y_{k + 1} = A_k x_k + a_{k + 1} v_k - \bar{c}_{k + 1} B^{-1} s_k, \quad k \inn{1}{T - 2},
\end{equation}
where $(\bar{c}_k)_{k \inn{1}{T - 1}}$ is a positive weight sequence. Outside the range $k \inn{1}{T - 2}$ we have the updates $y_1 = x_0$ and $y_T = v_{T - 1}$ as given by \eqref{label_047}. Note that \eqref{label_047} is a more efficient substitute for the equivalent update $y_T = x_{T - 1} - b_{T - 1} B^{-1} s_{T - 1}$, also found in \cite{ref_003}.

From \eqref{label_041} and \eqref{label_042} we have that
\begin{equation} \label{label_055}
\begin{gathered}
x_k - v_k = \sum_{i = 1}^{k}(c_i - \frac{1}{L_i} - a_i B_{i,k}) B^{-1} g_i = \sum_{i = 1}^{k}(c_i - \frac{1}{L_i} - a_i B_{i,T} + a_i B_{k,T}) B^{-1} g_i \\ \overset{\eqref{label_046}}{=} \sum_{i = 1}^{k} a_i B_{k,T} B^{-1} g_i \overset{\eqref{label_044}}{=} B_{k, T} B^{-1} s_k, \quad k \inn{1}{T - 2}.
\end{gathered}
\end{equation}
From line~\ref{label_021_y} in Algorithm~\ref{label_021} we have
\begin{equation} \label{label_054_c}
\begin{gathered}
A_{k + 1} y_{k + 1} = A_k x_k + a_{k + 1} v_k + a_{k + 1}(x_k - v_k) - A_{k + 1} b_k B^{-1} s_k \\
= A_k x_k + a_{k + 1} v_k - (A_{k + 1} b_k - a_{k + 1} B_{k, T}) B^{-1} s_k, \quad k \inn{1}{T - 2}.
\end{gathered}
\end{equation}
The weights are thus given by
\begin{equation} \label{label_057}
\begin{gathered}
\bar{c}_{k + 1} = A_{k + 1} b_k - a_{k + 1} B_{k, T} = (A_{k} + a_{k + 1}) b_k - a_{k + 1} (b_k + B_{k + 1, T}) \\\ = A_{k} b_k - a_{k + 1} B_{k + 1, T}, \quad k \inn{1}{T - 2}.
\end{gathered}
\end{equation}

We list this alternate form in Algorithm~\ref{label_021_diak}. It is the most restrictive because it requires $A_0$ to be supplied as an input parameter and the sequence $(b_k)_{k \inn{1}{T - 1}}$ needs to be precomputed as well.

\begin{algorithm}[h]
\caption{Two auxiliary sequence form of Algorithm~\ref{label_021}}
\label{label_021_diak}
\begin{algorithmic}[1]
\STATE \textbf{Input:} $x_0 \in \mathbb{E}$, $A_0 > 0$, $T \geq 2$
\STATE \textbf{Output:} $y_T = v_{T - 1}$, $A_{T - 1}$
\STATE Precompute the sequence $(b_k)_{k \inn{1}{T - 1}}$
\STATE $v_0 = x_0$
\STATE $s_0 = 0$
\FOR{$k = 0, \ldots{}, T - 2$}
\STATE Compute $a_{k + 1}$ \label{label_059}
\STATE $A_{k + 1} = A_k + a_{k + 1}$
\STATE $y_{k + 1} = \left\{\begin{array}{ll}x_0, &k = 0 \\
\frac{1}{A_{k + 1}} \left( A_k x_k + a_{k + 1} v_k - (A_k b_k - a_{k + 1} B_{k + 1, T}) B^{-1} s_k \right), & k \geq 1 \end{array}\right.$\label{label_060}
\STATE Based on $y_{k + 1}$ compute $L_{k + 1}$
\STATE $x_{k + 1} = T_{L_{k + 1}}(y_{k + 1})$
\STATE $g_{k + 1} = L_{k + 1} B (y_{k + 1} - x_{k + 1})$ \label{label_061}
\STATE $s_{k + 1} = s_k + a_{k + 1} g_{k + 1}$
\STATE $v_{k + 1} = v_k - \left(\frac{1}{L_{k + 1}} + a_{k + 1} B_{k + 1, T}\right) B^{-1} g_{k + 1}$
\ENDFOR
\end{algorithmic}
\end{algorithm}

\subsection{Convergence results} \label{label_062}

The template in Algorithm~\ref{label_021} is specific enough to allow us to obtain useful results that will be shared among the convergence analyses of \emph{all} instances, without requiring any knowledge on the problem. Key to our analysis are the following gradient aggregators
\begin{equation} \label{label_063}
G^{(1)}_T \defq \sum_{k = 0}^{T - 1} A_k \langle g_{k + 1}, x_k - y_{k + 1} \rangle, \quad
G^{(2)}_T \defq \sum_{k = 1}^{T - 1} a_k \langle g_{k}, x_T - y_{k} \rangle.
\end{equation}
We apply line~\ref{label_021_y} in Algorithm~\ref{label_021} to expand $G^{(1)}_T$ in \eqref{label_063} as
\begin{equation} \label{label_064}
G^{(1)}_T = \sum_{k = 1}^{T - 1} \langle g_{k + 1}, A_k b_k B^{-1} s_k \rangle \overset{\eqref{label_044}}{=} \sum_{i = 1}^{T} \sum_{j = i + 1}^{T} a_i A_{j - 1} b_{j - 1} \langle g_i, B^{-1} g_j \rangle.
\end{equation}
From \eqref{label_042} we have that
\begin{gather}
x_T - y_k = x_0 - (\mathcal{L}_T - \mathcal{L}_{k - 1}) - (\mathcal{S}_{T - 1} - \mathcal{S}_{k - 1}) , \quad k \inn{1}{T-1}. \label{label_065}
\end{gather}
Then, we can substitute the unrolled update in \eqref{label_065} to $G^{(2)}_T$ in \eqref{label_063}, yielding
\begin{equation} \label{label_066}
G^{(2)}_T = -\sum_{k = 1}^{T - 1} \langle g_{k}, a_k (\mathcal{L}_T - \mathcal{L}_{k - 1}) \rangle -\sum_{k = 1}^{T - 1} \langle g_{k}, a_k (\mathcal{S}_{T - 1} - \mathcal{S}_{k - 1}) \rangle.
\end{equation}
We expand two subexpressions in \eqref{label_066} and attain for all $k \inn{1}{T-1}$
\begin{gather}
\langle g_k, a_k (\mathcal{L}_T - \mathcal{L}_{k - 1}) \rangle = \frac{a_k}{L_k} \| g_k \|_*^2 + \sum_{i = k + 1}^{T} \frac{a_k}{L_i} \langle g_k, B^{-1} g_i \rangle , \label{label_067} \\
\mathcal{S}_{T - 1} - \mathcal{S}_{k - 1} = \sum_{i = k}^{T - 1} b_i B^{-1} s_i = \sum_{i = k}^{T - 1} b_i B^{-1} (s_i - s_{k - 1}) + \sum_{i = k}^{T - 1} b_i B^{-1} s_{k - 1}. \label{label_068}
\end{gather}
Further, we use the definition $B_{k,T}$ to express \eqref{label_068} as
\begin{equation} \label{label_069}
\begin{gathered}
\mathcal{S}_{T - 1} - \mathcal{S}_{k - 1} \overset{\eqref{label_044}}{=} \sum_{i = k}^{T - 1} b_i B^{-1} \sum_{j = k}^{i} a_j g_j + \sum_{i = k}^{T - 1} b_i B^{-1} \sum_{j = 1}^{k - 1} a_j g_j \\ \overset{\eqref{label_028}}{=} a_k B_{k,T} B^{-1} g_k + \sum_{i = k + 1}^{T - 1} a_i B_{i,T} B^{-1} g_i + \sum_{i = 1}^{k - 1} a_i B_{k,T} B^{-1} g_i, \quad k \inn{1}{T-1}.
\end{gathered}
\end{equation}
We apply $\langle g_k, . \rangle$ to both sides of \eqref{label_069} and have for all $k \inn{1}{T-1}$ that
\begin{equation} \label{label_070}
\begin{gathered}
\langle g_k, a_k (\mathcal{S}_{T - 1} - \mathcal{S}_{k - 1}) \rangle = a_k^2 B_{k,T} \| g_k \|_*^2 \\
+ \sum_{i = 1}^{k - 1} a_i a_k B_{k,T} \langle g_i, B^{-1} g_k \rangle
+ \sum_{i = k + 1}^{T - 1} a_k a_i B_{i,T} \langle g_k, B^{-1} g_i \rangle.
\end{gathered}
\end{equation}

We substitute \eqref{label_067} and \eqref{label_070} in \eqref{label_066} and add the result to \eqref{label_064} to obtain
\begin{equation} \label{label_071}
\begin{gathered}
G^{(1)}_T + G^{(2)}_T = \sum_{i = 1}^{T - 1} \sum_{j = i + 1}^{T} a_i A_{j - 1} b_{j - 1} \langle g_i, B^{-1} g_j \rangle - \sum_{k = 1}^{T - 1} \frac{a_k}{L_k} \| g_k \|_*^2 \\
- \sum_{i = 1}^{T - 1} \sum_{j = i + 1}^{T} \frac{a_i}{L_j} \langle g_i, B^{-1} g_j \rangle
- \sum_{k = 1}^{T - 1} a_k^2 B_{k, T} \| g_k \|_*^2 - \sum_{i = 1}^{T - 1} \sum_{j = i + 1}^{T - 1} 2 a_i a_j B_{j, T} \langle g_i, B^{-1} g_j \rangle.
\end{gathered}
\end{equation}
Rearranging terms in \eqref{label_071} using $B_{T, T} = 0$ gives the main technical result
\begin{equation} \label{label_072}
G^{(1)}_T + G^{(2)}_T = \sum_{k = 1}^{T - 1} \left( - \frac{a_k}{L_k} - a_k^2 B_{k, T} \right) \| g_k \|_*^2 + D_{T} \overset{\eqref{label_046}}{=} -\sum_{k = 1}^{T - 1} a_k c_k \| g_k \|_*^2 + D_{T},
\end{equation}
where for all $T \geq 2$ we define $D_{T}$ as
\begin{gather}
D_{T} \defq \sum_{i = 1}^{T - 1} \sum_{j = i + 1}^{T - 1} a_i \left( A_{j - 1} b_{j - 1} - \frac{1}{L_j} - 2 a_j B_{j, T} \right) \langle g_i, B^{-1} g_j \rangle \nonumber \\
+ \sum_{i = 1}^{T - 1} a_i \left( A_{T - 1} b_{T - 1} - \frac{1}{L_T} \right) \langle g_i, B^{-1} g_T \rangle \\
\overset{\eqref{label_044}}{=} \sum_{k = 1}^{T - 2} \left( A_{k} b_{k} - \frac{1}{L_{k + 1}} - 2 a_{k + 1} B_{k + 1, T} \right) \langle s_k, B^{-1} g_{k + 1} \rangle \nonumber\\ + \left( A_{T - 1} b_{T - 1} - \frac{1}{L_{T}} \right) \langle s_{T - 1}, B^{-1} g_T \rangle.
\end{gather}

\section{Revisiting OGM-G} With the template and auxiliary results defined, we next study whether OGM-G is compatible with our framework, and if so, what we can gain from this new perspective.

\subsection{Building a method}
Let us first consider constructing an algorithm applicable to smooth unconstrained problems based on the bounds in \eqref{label_002}. These bounds are not known to allow for a line-search procedure when $x$ is an unknown or future point so we assume the algorithm knows $L$ and set $L_k = L$ for all $k \inn{1}{T}$.

Applying \eqref{label_002} with $x = y_k$ and $y = y_{k + 1}$ gives
\begin{equation} \label{label_073}
f(y_k) \geq f(y_{k + 1}) + \langle g_{k + 1}, x_k - y_{k + 1} \rangle + \frac{1}{2 L} \| g_k \|_*^2 + \frac{1}{2 L} \| g_{k + 1} \|_*^2, \quad k \inn{1}{T - 1}.
\end{equation}
Multiplying every \eqref{label_073} with $A_k$ for $k \inn{1}{T-1}$ and adding them together gives
\begin{equation} \label{label_074}
\begin{gathered}
A_1 f(y_1) + \sum_{k = 2}^{T - 1} a_k f(y_k) \geq A_{T - 1} f(y_T) \\
+ \frac{A_1}{2 L} \| g_{1} \|_*^2 + \sum_{k = 1}^{T - 2} \frac{A_k + A_{k + 1}}{2 L} \| g_{k + 1} \|_*^2 + \frac{A_{T - 1}}{2 L} \| g_{T} \|_*^2 + G^{(1)}_T.
\end{gathered}
\end{equation}
Applying \eqref{label_002} this time with $x = y_T$ and $y = y_{k + 1}$ yields
\begin{equation} \label{label_075}
f(y_T) \geq f(y_{k + 1}) + \langle g_{k + 1}, x_T - y_{k + 1} \rangle + \frac{1}{2 L} \| g_T \|_*^2 + \frac{1}{2 L} \| g_{k + 1} \|_*^2, \quad k \inn{0}{T-2}.
\end{equation}
Multiplying every \eqref{label_075} with $a_{k + 1}$ for $k \inn{0}{T-2}$ and taking the summation gives
\begin{equation} \label{label_076}
(A_{T - 1} - A_0) f(y_T) \geq \sum_{k = 1}^{T - 1} a_k f(y_k) + \sum_{k = 0}^{T - 2} \frac{a_{k + 1}}{2 L} \| g_{k + 1} \|_*^2 + \frac{A_{T - 1} - A_0}{2 L} \| g_T \|_*^2 + G^{(2)}_T.
\end{equation}
Note that the ranges of $k$ in \eqref{label_073} and \eqref{label_075} are the widest that avoid redundancies.

Adding together \eqref{label_074} with \eqref{label_076} and canceling matching terms while taking into account that $A_1 f(y_1) = A_0 f(x_0) + a_1 f(y_1)$ we obtain
\begin{equation} \label{label_077}
\begin{gathered}
A_0 (f(x_0) - f(y_T)) \geq G^{(1)}_T + G^{(2)}_T \\ + \frac{A_1 + a_1}{2 L} \| g_{1} \|_*^2 + \sum_{k = 2}^{T - 1} \frac{A_k}{L} \| g_{k} \|_*^2 + \left(\frac{A_{T - 1}}{L} - \frac{A_0}{2 L} \right) \| g_{T} \|_*^2.
\end{gathered}
\end{equation}
The gradient step point $x_T$ satisfies the well-known descent rule~\cite{ref_019,ref_020}
\begin{equation} \label{label_078}
f(x_T) \leq f(y_T) - \frac{1}{2 L} \| g_{T} \|_*^2.
\end{equation}
Multiplying both sides of \eqref{label_078} with $\frac{A_0}{2 L}$ and adding the result to \eqref{label_077} gives
\begin{equation} \label{label_079}
A_0 (f(x_0) - f(x_T)) \geq G^{(1)}_T + G^{(2)}_T + \frac{A_1 + a_1}{2 L} \| g_{1} \|_*^2 + \sum_{k = 2}^{T - 1} \frac{A_k}{L} \| g_{k} \|_*^2 + \frac{A_{T - 1}}{L} \| g_{T} \|_*^2.
\end{equation}
Expanding $G^{(1)}_T + G^{(2)}_T$ using \eqref{label_072} we finally get
\begin{equation} \label{eq:ogm-g-main}
\begin{gathered}
A_0 (f(x_0) - f(x_T)) \geq \left(\frac{A_0}{2 L} - a_1^2 B_{1, T} \right) \| g_{1} \|_*^2 + \sum_{k = 2}^{T - 1} \left(\frac{A_{k - 1}}{L} - a_k^2 B_{k, T} \right)\| g_{k} \|_*^2 \\
+ \frac{A_{T - 1}}{L} \| g_{T} \|_*^2 + \sum_{k = 1}^{T - 2} \left( A_{k} b_{k} - \frac{1}{L} - 2 a_{k + 1} B_{k + 1, T} \right) \langle s_k, B^{-1} g_{k + 1} \rangle \\
+ \left( A_{T - 1} b_{T - 1} - \frac{1}{L} \right) \langle s_{T - 1}, B^{-1} g_{T} \rangle.
\end{gathered}
\end{equation}
Our aim is to maximize $\|g_T\|$ subject to $f(x_0) - f(x_T)$, irrespective of the algorithmic state. The most aggressive means of achieving this is by enforcing in \eqref{eq:ogm-g-main} the following equalities:
\begin{align}
\frac{A_0}{2 L} &= a_1^2 B_{1, T}, \label{label_080} \\
\frac{A_{k - 1}}{L} &= a_k^2 B_{k, T}, \quad k \inn{2}{T - 1}, \label{label_081} \\
A_{k} b_{k} &= \frac{1}{L} + 2 a_{k + 1} B_{k + 1, T}, \quad k \inn{1}{T - 1}. \label{label_082} \\
A_{T - 1} b_{T - 1} &= \frac{1}{L}. \label{label_083}
\end{align}
Under the assumptions \eqref{label_080}, \eqref{label_081}, \eqref{label_082} and \eqref{label_083}, we can refactor \eqref{eq:ogm-g-main} to yield
\begin{equation} \label{label_084}
\| g_T \|^2 \leq \frac{L A_0}{A_{T - 1}} (f(x_0) - f(x_T)) \leq \frac{L A_0}{A_{T - 1}} (f(x_0) - f(x^*)).
\end{equation}

\begin{proposition}
The conditions in \eqref{label_080}, \eqref{label_081} and \eqref{label_082} admit the following solution:
\begin{align}
A_k A_{k + 1}^2 &= A_{T - 1} a_{k + 1}^2, \quad k \inn{1}{T - 2}, \label{label_081_rec}\\
A_0 A_1^2 &= 2 A_{T - 1} a_1^2, \label{label_080_rec} \\
B_{k, T} &= \frac{A_{T - 1}}{L A_k^2}, \quad k \inn{1}{T - 1}, \label{label_087}\\
b_k &= \frac{1}{L} \left(\frac{1}{A_k} + \frac{2}{a_{k + 1}} \right), \quad k \inn{1}{T - 2}, \label{label_088}
\end{align}
\end{proposition}
\begin{proof}
First we restrict ourselves $k = T - 1$, unless specified otherwise. Combining \eqref{label_083} with $B_{T - 1, T} = b_{T - 1}$ implies \eqref{label_087}. Substituting \eqref{label_087} in \eqref{label_081} gives \eqref{label_081_rec} for $k = T - 2$. Under the restriction $k = T - 2$, \eqref{label_081_rec} and \eqref{label_082} imply \eqref{label_088}.

Now, we proceed by induction in reverse for \eqref{label_081_rec}, \eqref{label_087} and \eqref{label_088} . We have shown that \eqref{label_081_rec}, \eqref{label_087} and \eqref{label_088} hold for the last index in their respective ranges.

We assume that for a given $\bar{k}$, \eqref{label_081_rec}, \eqref{label_087} and \eqref{label_088} hold for $\bar{k} + 1$, $\bar{k} + 2$ and $\bar{k} + 1$, respectively. We have
\begin{equation} \label{label_089}
B_{\bar{k} + 1, T} = b_{\bar{k} + 1} + B_{\bar{k} + 2, T} = \frac{1}{L}\left(\frac{1}{A_{\bar{k} + 1}} + \frac{2}{a_{\bar{k} + 2}}\right) + \frac{A_{T - 1}}{L A_{\bar{k} + 2}^2}.
\end{equation}
Applying \eqref{label_081_rec} to \eqref{label_089} we get
\begin{equation} \label{label_090}
\begin{gathered}
B_{\bar{k} + 1, T} = \frac{1}{L}\left(\frac{A_{\bar{k} + 2}^2}{A_{\bar{k} + 1} A_{\bar{k} + 2}^2} + \frac{2 A_{T - 1} a_{\bar{k} + 2}}{A_{\bar{k} + 1} A_{\bar{k} + 2}^2} \right) + \frac{A_{T - 1}}{L A_{\bar{k} + 2}^2} \\
\overset{\eqref{label_081_rec}}{=} \frac{1}{L}\left(\frac{A_{T - 1} a_{\bar{k} + 2}^2}{A_{\bar{k} + 1}^2 A_{\bar{k} + 2}^2} + \frac{2 A_{T - 1} a_{\bar{k} + 2} A_{\bar{k} + 1}}{A_{\bar{k} + 1}^2 A_{\bar{k} + 2}^2} \right) + \frac{A_{T - 1} A_{\bar{k} + 1}^2}{L A_{\bar{k} + 1}^2 A_{\bar{k} + 2}^2} = \frac{A_{T - 1}}{L A_{\bar{k} + 1}^2}.
\end{gathered}
\end{equation}
In \eqref{label_090} we have just proven that \eqref{label_087} holds for $k = \bar{k} + 1$. Further applying \eqref{label_081} with $k = \bar{k} + 1$ gives \eqref{label_081_rec} with $k = \bar{k}$. Also, putting together \eqref{label_082} with $k = \bar{k}$, \eqref{label_081_rec} with $k = \bar{k}$ and \eqref{label_090} yields \eqref{label_088} with $k = \bar{k}$. We have just shown by induction that \eqref{label_081_rec}, \eqref{label_087} and \eqref{label_088} hold for the ranges $k \inn{1}{T - 2}$, $k \inn{2}{T - 1}$ and $k \inn{1}{T - 2}$, respectively.

The proof for $B_{1,T}$ follows in the same way as in \eqref{label_089} and \eqref{label_090}.
Finally, applying \eqref{label_080} to \eqref{label_087} with $k = 1$ produces \eqref{label_080_rec}.
\end{proof}
\subsection{Simplifying the updates}
Taking the square root in \eqref{label_081_rec} and multiplying both sides with $\frac{\sqrt{A_{T - 1}}}{A_k A_{k + 1}}$ yields
\begin{equation} \label{label_091}
\sqrt{\frac{A_{T - 1}}{A_k}} = \frac{A_{T - 1}}{A_k} - \frac{A_{T - 1}}{A_{k + 1}}, \quad k \inn{1}{T - 2}.
\end{equation}
Performing the same operation with $k = 0$ in \eqref{label_080_rec} leads to
\begin{equation} \label{label_092}
\sqrt{\frac{A_{T - 1}}{A_0}} = \sqrt{2} \frac{A_{T - 1}}{A_0} - \sqrt{2} \frac{A_{T - 1}}{A_1}.
\end{equation}
To adapt our analysis to the notation found in \cite{ref_011}, we define the sequence $(\theta_{k, T})_{k \inn{0}{T}}$ for any $T \geq 2$ as
\begin{equation} \label{label_093}
\theta_{k, T} = \left\{\begin{array}{ll}
\frac{1}{2}\left( 1 + \sqrt{1 + 8 \theta_{k + 1, T}^2} \right), &k = 0 \\
\frac{1}{2}\left( 1 + \sqrt{1 + 4 \theta_{k + 1, T}^2} \right), &k \inn{1}{T-2} \\
1, &k = T - 1 \\
0, &k = T
\end{array} \right..
\end{equation}
The definition in \eqref{label_093} ensures the recursions
\begin{equation} \label{label_094}
\theta_{0, T} = \theta_{0, T}^2 - 2 \theta_{1, T}^2, \quad \theta_{k, T} = \theta_{k, T}^2 - \theta_{k + 1, T}^2, \quad k \inn{1}{T - 1}.
\end{equation}
From \eqref{label_091}, \eqref{label_092}, \eqref{label_093} and \eqref{label_094} we have that
\begin{equation} \label{label_095}
A_0 = 2 \frac{A_{T-1}}{\theta_{0,T}^2}, \quad A_k = \frac{A_{T-1}}{\theta_{k,T}^2}, \quad k \inn{1}{T - 1}.
\end{equation}
Taking the differences in \eqref{label_095} using \eqref{label_094} we obtain the simple formula
\begin{equation} \label{label_096}
a_{k + 1} = \frac{A_{T-1}}{\theta_{k,T} \theta_{k + 1,T}^2}, \quad k \inn{0}{T - 2}.
\end{equation}
Substituting \eqref{label_095} and \eqref{label_096} into \eqref{label_088} and \eqref{label_083} yields a unified expression for all the steps $b_k$ taking the form
\begin{equation} \label{label_097}
b_{k} = \frac{1}{L A_{T-1}} \theta_{k,T} (\theta_{k,T} + 2 \theta_{k + 1,T}^2) \overset{\eqref{label_094}}{=} \frac{1}{L A_{T-1}} \theta_{k,T}^2 (2 \theta_{k,T} - 1), \quad k \inn{1}{T - 1}.
\end{equation}
Further using \eqref{label_087}, we obtain a single expression for the reverse accumulator and an alternative formula for $b_k$, respectively given by
\begin{equation} \label{label_098}
B_{k,T} \overset{\eqref{label_095}}{=} \frac{\theta_{k,T}^4}{L A_{T - 1}}, \quad b_k \overset{\eqref{label_028}}{=} B_{k + 1,T} - B_{k,T} = \frac{\theta_{k + 1,T}^4 - \theta_{k,T}^4}{L A_{T-1}}, \quad k \inn{1}{T - 1}.
\end{equation}
Instantiating the template in Algorithm~\ref{label_021} with \eqref{label_096} and \eqref{label_097}, noticing that all $A_{T - 1}$ terms cancel, we obtain the method listed in Algorithm~\ref{label_099}.

\begin{algorithm}[h]
\caption{An optimal method for minimizing the gradient norm}
\label{label_099}
\begin{algorithmic}[1]
\STATE \textbf{Input:} $x_0 \in \mathbb{E}$, $T \geq 2$
\STATE \textbf{Output:} $y_T$, $g_T$, optionally $x_T$
\STATE Precompute the entire sequence $(\theta_{k, T})_{k \inn{0}{T}}$ using \eqref{label_093}
\STATE $s_0 = 0$
\FOR{$k = 0, \ldots{}, T - 1$}
\STATE $y_{k + 1} = x_k - \frac{1}{L} \theta_{k,T}^2 (2 \theta_{k,T} - 1) B^{-1} s_k$ \label{label_099_y}
\STATE $g_{k + 1} = f'(y_{k + 1})$
\STATE $x_{k + 1} = y_{k + 1} - \frac{1}{L} B^{-1} g_{k + 1}$
\IF{$k = T - 1$}
\STATE break from loop
\ENDIF
\STATE $s_{k + 1} = s_k + \frac{1}{\theta_{k,T} \theta_{k + 1,T}^2} g_{k + 1}$
\ENDFOR
\end{algorithmic}
\end{algorithm}
\subsection{Relationship to OGM-G} We have just formulated an algorithm for smooth unconstrained gradient norm minimization with the largest last iterate guarantees allowed by our analysis. The update rules in Algorithm~\ref{label_099} clearly differ from all previously introduced methods. We therefore need to explore its equivalent forms and compare them with all previously introduced forms of the method with the highest known guarantees, OGM-G.
\begin{proposition}
Algorithm~\ref{label_099} yields the same output $y_T$ as the original extrapolated form of OGM-G in \cite{ref_011}.
\end{proposition}
\begin{proof}
OGM-G updates the sequence $(y_k)_{k \inn{1}{T}}$ starting with $y_1 = x_0$ and continuing for all $k \inn{1}{T - 1}$ with
\begin{equation} \label{label_101}
y_{k + 1} = x_k + \frac{(\theta_{k - 1,T} - 1) (2 \theta_{k,T} - 1)}{\theta_{k - 1,T} (2 \theta_{k - 1,T} - 1)}(x_k - x_{k - 1}) + \frac{2 \theta_{k,T} - 1}{2 \theta_{k - 1,T} - 1} (x_k - y_k).
\end{equation}
When $k = 1$, \eqref{label_101} simplifies to
\begin{equation} \label{label_102}
\begin{aligned}
y_{2} & = x_1 - \left(1 + \frac{\theta_{0,T} - 1}{\theta_{0,T}}\right) \frac{2 \theta_{1,T} - 1}{2 \theta_{0,T} - 1} \frac{1}{L} B^{-1} g_1 \\
& = x_1 - \frac{2 \theta_{0,T} - 1}{\theta_{0,T}} \frac{2 \theta_{1,T} - 1}{2 \theta_{0,T} - 1} \frac{1}{L} B^{-1} g_1 = x_1 - \frac{2 \theta_{1,T} - 1}{L \theta_{0,T}} B^{-1} g_1.
\end{aligned}
\end{equation}
The corresponding update in Algorithm~\ref{label_099} is
\begin{equation} \label{label_103}
\begin{aligned}
y_{2} &= x_1 - b_1 a_1 B^{-1} g_1 = x_1 - \frac{1}{L} \theta_{1,T}^2 (2 \theta_{1,T} - 1) \frac{1}{\theta_{0,T} \theta_{1,T}^2} B^{-1} g_1 \\ & = x_1 - \frac{2 \theta_{1,T} - 1}{L \theta_{0,T}} B^{-1} g_1.
\end{aligned}
\end{equation}
Given the equivalence of \eqref{label_102} and \eqref{label_103}, in the following we need only consider the case when $k \inn{2}{T - 1}$, where the extrapolation and correction factors are respectively given by
\begin{equation} \label{label_104}
\begin{aligned}
\mathcal{E}_k = \frac{b_k}{b_{k - 1}} &= \frac{\theta_{k, T}^2 (2 \theta_{k, T} - 1)}{\theta_{k - 1, T}^2 (2 \theta_{k - 1, T} -1)} \overset{\eqref{label_094}}{=} \frac{(\theta_{k - 1, T}^2 - \theta_{k - 1, T}) (2 \theta_{k, T} - 1)}{\theta_{k - 1, T}^2 (2 \theta_{k - 1, T} -1)} \\ &= \frac{(\theta_{k - 1,T} - 1) (2 \theta_{k,T} - 1)}{\theta_{k - 1,T} (2 \theta_{k - 1,T} - 1)},
\end{aligned}
\end{equation}
\begin{equation} \label{label_105}
\begin{gathered}
\mathcal{C}_k = L a_k b_k - \frac{b_k}{b_{k - 1}} =
\frac{\theta_{k, T} (2 \theta_{k, T} - 1)}{\theta_{k - 1, T}} - \frac{(\theta_{k - 1,T} - 1) (2 \theta_{k,T} - 1)}{\theta_{k - 1,T} (2 \theta_{k - 1,T} - 1)}
= \frac{2 \theta_{k,T} - 1}{2 \theta_{k - 1,T} - 1}.
\end{gathered}
\end{equation}
The factors in \eqref{label_104} and \eqref{label_105} match those of OGM-G.
\end{proof}

For completeness, we write down the extrapolation based form of OGM-G using our notation in Algorithm~\ref{label_099-extra}.
\begin{algorithm}[h]
\caption{Extrapolation based form of Algorithm~\ref{label_099}}
\label{label_099-extra}
\begin{algorithmic}[1]
\STATE \textbf{Input:} $x_0 \in \mathbb{E}$, $T \geq 2$
\STATE \textbf{Output:} $y_T$, $g_T$, optionally $x_T$
\STATE Precompute the entire sequence $(\theta_{k, T})_{k \inn{0}{T}}$ using \eqref{label_093}
\FOR{$k = 0, \ldots{}, T - 1$}
\STATE $y_{k + 1} = \left\{\begin{array}{ll} x_0, & k = 0 \\
x_k + \frac{(\theta_{k - 1,T} - 1) (2 \theta_{k,T} - 1)}{\theta_{k - 1,T} (2 \theta_{k - 1,T} - 1)}(x_k - x_{k - 1}) + \frac{2 \theta_{k,T} - 1}{2 \theta_{k - 1,T} - 1} (x_k - y_k), & k \geq 1\end{array}\right.$ \label{label_099-extra_y}
\STATE $g_{k + 1} = f'(y_{k + 1})$
\STATE $x_{k + 1} = y_{k + 1} - \frac{1}{L} B^{-1} g_{k + 1}$
\ENDFOR
\end{algorithmic}
\end{algorithm}

We next write Algorithm~\ref{label_099} in auxiliary sequence based form using
\begin{equation} \label{label_106}
d_k \overset{\eqref{label_046}}{=} \frac{b_{k, T}}{B_{k, T}} \overset{\eqref{label_028}}{=}1 - \frac{B_{k + 1, T}}{B_{k, T}} \overset{\eqref{label_098}}{=} 1 - \frac{\theta_{k + 1, T}^4}{\theta_{k, T}^4}, \quad k \inn{1}{T - 2}.
\end{equation}
Because $v_0 = x_0$, we can extend \eqref{label_106} to the first iteration $k = 0$ without altering the update in line~\ref{label_050} of Algorithm~\ref{label_048}. The auxiliary sequence weights also become
\begin{equation}
c_{k + 1} \overset{\eqref{label_046}}{=} \frac{1}{L} + a_{k + 1} B_{k + 1, T} \overset{\eqref{label_096}}{=} \frac{1}{L} + \frac{A_{T - 1}}{\theta_{k, T} \theta_{k + 1, T}^2} \frac{\theta_{k + 1, T}^4}{L A_{T - 1}} = \frac{\theta_{k,T} + \theta_{k + 1,T}^2}{L \theta_{k,T}}.
\end{equation}
In light of \eqref{label_094} we have
\begin{equation} \label{label_107}
c_1 = \frac{\theta_{0,T} + \theta_{1,T}^2}{L \theta_{0,T}} = \frac{\frac{1}{2}(\theta_{0,T}^2 - \theta_{0,T}) + \theta_{0,T}}{L \theta_{0,T}} = \frac{\theta_{0,T}+1}{2 L}.
\end{equation}
whereas
\begin{equation} \label{label_108}
c_{k + 1} = \frac{\theta_{k,T} + \theta_{k + 1,T}^2}{L \theta_{k,T}} = \frac{\theta_{k,T}^2}{L \theta_{k,T}} = \frac{\theta_{k,T}}{L}, \quad k \inn{1}{T - 1}.
\end{equation}
We can thus list this one auxiliary sequence form in Algorithm~\ref{label_099_aux1}. The update rules in Algorithm~\ref{label_099_aux1} indeed match those listed in \cite{ref_014}.

\begin{algorithm}[h]
\caption{Auxiliary sequence form of Algorithm~\ref{label_099}}
\label{label_099_aux1}
\begin{algorithmic}[1]
\STATE \textbf{Input:} $x_0 \in \mathbb{E}$, $T \geq 2$
\STATE \textbf{Output:} $y_T = v_{T - 1}$
\STATE Precompute the entire sequence $(\theta_{k, T})_{k \inn{0}{T}}$ using \eqref{label_093}
\STATE $v_0 = x_0$ \label{label_099_aux1_s0}
\FOR{$k = 0, \ldots{}, T - 2$}
\STATE $y_{k + 1} = \frac{\theta_{k + 1, T}^4}{\theta_{k, T}^4} x_k + \left(1 - \frac{\theta_{k + 1, T}^4}{\theta_{k, T}^4} \right) v_k$ \label{label_099_aux1_y}
\STATE $g_{k + 1} = f'(y_{k + 1})$
\STATE $x_{k + 1} = y_{k + 1} - \frac{1}{L} B^{-1} g_{k + 1}$
\STATE $v_{k + 1} = \left\{\begin{array}{ll} v_0 - \frac{\theta_{0,T} + 1}{2 L} B^{-1} g_{1}, & k = 0 \\ v_k - \frac{\theta_{k,T}}{L} B^{-1} g_{k + 1}, & k \geq 1\end{array}\right.$
\ENDFOR
\end{algorithmic}
\end{algorithm}

The final form for Algorithm~\ref{label_099} that we explore is the two auxiliary sequence one introduced in \cite{ref_003}. We adopt the notation in \cite{ref_003} and design the update rules to use the sequences $a_k$ and $A_k$. We thus have
\begin{equation} \label{label_112}
\begin{gathered}
\bar{c}_{k + 1} \overset{\eqref{label_057}}{=} A_k b_k - a_{k + 1} B_{k + 1, T} = \frac{A_{T - 1}}{\theta_{k, T}^2} \frac{\theta_{k, T}^2 (2 \theta_{k, T} - 1)}{L A_{T - 1}} - \frac{A_{T - 1}}{\theta_{k, T} \theta_{k + 1, T}^2}\frac{\theta_{k + 1, T}^4}{L A_{T - 1}}
\\ = \frac{1}{L}\left(2 \theta_{k, T} - 1 - \frac{\theta_{k + 1, T}^2}{\theta_{k, T}}\right)
\overset{\eqref{label_094}}{=} \frac{\theta_k}{L} \overset{\eqref{label_095}}{=} \frac{A_{k + 1}}{L a_{k + 1}}, \quad k \inn{1}{T - 2}.
\end{gathered}
\end{equation}
The sequence $v_k$ is updated in the same way as before, but we now write the step sizes in \eqref{label_107} and \eqref{label_108} as
\begin{equation} \label{label_113}
c_1 \overset{\eqref{label_107}}{=} \frac{1}{2L}\left(\frac{A_1}{a_1} + 1 \right) = \frac{A_1 + a_1}{2 L a_{1}},\quad c_{k + 1} \overset{\eqref{label_108}}{=} \frac{A_{k + 1}}{L a_{k + 1}}, \quad k \inn{1}{T - 2}.
\end{equation}
Substituting the expressions for $A_0$ in \eqref{label_095}, the weights $a_{k + 1}$ in \eqref{label_096} as well as for $c_{k + 1}$ and $\bar{c}_{k + 1}$ in \eqref{label_113} and \eqref{label_112}, respectively, and extending the update in \eqref{label_054} to $k = 0$ yields the form listed in Algorithm~\ref{label_114}. Our Algorithm~\ref{label_114} notably differs from OGM-G-DW~\cite{ref_003} in the way the first iteration is handled, which explains why OGM-G has a worst-case rate that is higher by a factor of $4$ than OGM-G-DW.

Note that altering value of $A_{T - 1}$ does not change the behavior of Algorithm~\ref{label_114} in any way and it is best to set $A_{T - 1} = 1$. However, for compatibility with \cite{ref_003}, we leave $A_{T - 1}$ as an optional input parameter.

\begin{algorithm}[h]
\caption{Two auxiliary sequence form of Algorithm~\ref{label_099}}
\label{label_114}
\begin{algorithmic}[1]
\STATE \textbf{Input:} $x_0 \in \mathbb{E}$, $T \geq 2$, optionally $A_{T - 1} > 0$ or otherwise set $A_{T - 1} = 1$
\STATE \textbf{Output:} $y_T = v_{T - 1}$
\STATE Precompute the entire sequence $(\theta_{k, T})_{k \inn{0}{T}}$ using \eqref{label_093}
\STATE $A_0 = \frac{2 A_{T - 1}}{\theta_{0, T}^2}$
\STATE $v_0 = x_0$
\STATE $s_0 = 0$
\FOR{$k = 0, \ldots{}, T - 2$}
\STATE $a_{k + 1} = \frac{A_{T - 1}}{\theta_{k, T} \theta_{k + 1, T}^2}$
\STATE $A_{k + 1} = A_k + a_{k + 1}$
\STATE $y_{k + 1} = \frac{1}{A_{k + 1}} (A_k x_k + a_{k + 1} v_k) - \frac{1}{L a_{k + 1}} s_k $\label{label_114_y}
\STATE $g_{k + 1} = f'(y_{k + 1})$
\STATE $x_{k + 1} = y_{k + 1} - \frac{1}{L} B^{-1} g_{k + 1}$
\STATE $s_{k + 1} = s_k + a_{k + 1} g_{k + 1}$
\STATE $v_{k + 1} = \left\{\begin{array}{ll} v_0 - \frac{A_1 + a_1}{2 L a_1} g_{1}, & k = 0 \\ v_k - \frac{A_{k + 1}}{L a_{k + 1}} g_{k + 1}, & k \geq 1\end{array}\right.$
\ENDFOR
\end{algorithmic}
\end{algorithm}

\subsection{Convergence analysis} \label{label_116}

To establish a worst-case rate for Algorithm~\ref{label_099} as a simple non-recursive expression, we first need to establish some limits on the behavior of the sequence $(\theta_{k, T})_{k \inn{0}{T}}$.

\begin{lemma} \label{label_117}
The sequence $(\theta_{k, T})_{k \inn{0}{T}}$ can be bounded as follows:
\begin{gather}
\theta_{k, T} > \theta_{k + l, T} + \frac{l}{2}, \quad k \inn{0}{T - 1}, \quad l \inn{1}{T - k} \label{label_118} \\
\theta_{0, T} > \frac{T}{\sqrt{2}} + \frac{1}{2}, \label{label_119} \\
2 \theta_{k, T}^2 < \theta_{0, T} \left( \theta_{0, T} - \frac{k}{\sqrt{2}} \right), \quad k \inn{1}{T - 1}. \label{label_120}
\end{gather}
\end{lemma}
\begin{proof}
For all $k \inn{0}{T-2}$ we have
\begin{equation} \label{label_121}
\theta_{k, T} \overset{\eqref{label_093}}\geq \frac{1}{2}\left( 1 + \sqrt{1 + 4 \theta_{k + 1, T}^2} \right) > \frac{1}{2}\left( 1 + \sqrt{4 \theta_{k + 1, T}^2} \right) = \theta_{k + 1, T} + \frac{1}{2}.
\end{equation}
Because $\theta_{T-1,T} = \theta_{T,T} + 1$ we have that \eqref{label_121} holds also for $k = T - 1$.
Iterating \eqref{label_121} over this extended range yields \eqref{label_118}.

When $T \geq 3$, taking \eqref{label_118} with $k = 1$ and $l = T - 2$ we get
\begin{equation}
\theta_{1, T} \geq \theta_{T - 1, T} + \frac{T - 2}{2} = \frac{T}{2}.
\end{equation}
When $T = 2$ we have $\theta_{1, T} = 1 = \frac{T}{2}$. Hence for all $T \geq 2$ it holds that
\begin{equation} \label{label_122}
\theta_{0, T} = \frac{1}{2}\left( 1 + \sqrt{1 + 8 \theta_{1, T}^2} \right) > \frac{1}{2} + \sqrt{2} \theta_{1, T} \geq \frac{T}{\sqrt{2}} + \frac{1}{2}.
\end{equation}

Similarly, for all $k \geq 1$ we have
\begin{equation} \label{label_123}
\theta_{0, T} \overset{\eqref{label_122}}{>} \sqrt{2} \theta_{1, T} + \frac{1}{2} \overset{\eqref{label_118}}{\geq} \sqrt{2} \left( \theta_{k, T} + \frac{k - 1}{2} \right) + \frac{1}{2} = \sqrt{2} \theta_{k,T} + \frac{k}{\sqrt{2}} - \delta,
\end{equation}
where $\delta \defq \frac{\sqrt{2} - 1}{2} \approx 0.20711$.

For all $k \inn{1}{T - 1}$ we can ensure that
\begin{equation} \label{label_124}
\theta_{0, T} > \frac{T}{\sqrt{2}} > \frac{k}{\sqrt{2}}, \quad
k \geq 1 > (2 \sqrt{2} + 1) \delta, \quad
\theta_{k, T} \geq 1 > \left(\frac{1}{\sqrt{2}} + 1 \right) \delta.
\end{equation}
We finally obtain
\begin{equation}
\begin{gathered}
\theta_{0, T} \left( \theta_{0, T} - \frac{k}{\sqrt{2}} \right) \overset{\eqref{label_123}}{>}
\left( \sqrt{2} \theta_{k,T} + \frac{k}{\sqrt{2}} - \delta \right) \left( \sqrt{2} \theta_{k,T} - \delta \right) \\
= 2 \theta_{k, T}^2 + (k - 2 \sqrt{2} \delta)\left(\theta_{k,T} - \frac{\delta}{\sqrt{2}}\right) - \delta^2 \overset{\eqref{label_124}}{>} 2 \theta_{k, T}^2.
\end{gathered}
\end{equation}
\end{proof}
With the results in Lemma~\ref{label_117} we can analyze the worst-case behavior of OGM-G. Note that to produce $y_T$ we need to perform only $T-1$ gradient evaluations. Our notation thus differs from \cite{ref_011}, which explains why we have $T$ instead of $T + 1$ on the right-hand side of \eqref{label_119}. From \eqref{label_084} we obtain the worst-case rate
\begin{equation} \label{label_125}
\| g_T \|^2 \overset{\eqref{label_084}}{\leq} \frac{L A_0}{A_{T - 1}} (f(x_0) - f(x_T)) \overset{\eqref{label_095}}{=} \frac{2 L}{\theta_{0, T}^2} (f(x_0) - f(x_T)) \overset{\eqref{label_119}}{<} \frac{4 L}{T^2} (f(x_0) - f(x_T)).
\end{equation}
The result in \eqref{label_125} is consistent with the previous findings in \cite{ref_011}, confirming the equivalence between Algorithm~\ref{label_099} and OGM-G.

Our construction of Algorithm~\ref{label_099} relied on providing the best rate for the last iterate $y_T$ as given by \eqref{label_084}. It is worth inquiring how this method behaves in all the iterations leading up to the final result. Interestingly, the answer takes the following simple form.

\begin{proposition} \label{prop:midway-smooth}
The state variables of Algorithm~\ref{label_099} satisfy at runtime
\begin{equation} \label{label_044_B_k}
B_{k, T} \| s_{k} \|_*^2 \leq A_0 (f(x_0) - f(x_k)), \quad k \inn{1}{T-1}.
\end{equation}
\end{proposition}
\begin{proof}
When $k = 1$, \eqref{label_080} renders \eqref{label_044_B_k} equivalent to \eqref{label_029}. When $k \inn{2}{T-1}$, substituting $T$ with $k$ in \eqref{eq:ogm-g-main} and applying $B_{i, k} = B_{i, T} - B_{k, T}$ we get
\begin{equation} \label{eq:ocgm-101}
\begin{gathered}
A_0 (f(x_0) - f(x_k)) \geq \left(\frac{A_0}{2 L} - a_1^2 (B_{1, T} - B_{k, T}) \right) \| g_{1} \|_*^2 \\
+ \sum_{i = 2}^{k - 1} \left( \frac{A_{i - 1}}{L} - a_i^2 (B_{i, T} - B_{k, T}) \right) \| g_{i} \|_*^2
+ \left(\frac{A_{k - 1}}{L} - a_k^2 (B_{i, T} - B_{k, T}) \right)\| g_{k} \|_*^2
\\ + \sum_{i = 1}^{k - 2} \left( A_{i} b_{i} - \frac{1}{L} - 2 a_{i + 1} (B_{i + 1, T} - B_{k, T}) \right) \langle s_i, B^{-1} g_{i + 1} \rangle \\
+ \left( A_{k - 1} b_{k - 1} - \frac{1}{L} - 2 a_k B_{k, T} + 2 a_k B_{k, T} \right) \langle s_{k - 1}, B^{-1} g_{k} \rangle, \quad k \inn{2}{T-1}.
\end{gathered}
\end{equation}
Taking into account that the weights and step sizes satisfy the equalities \eqref{label_080}, \eqref{label_081} and \eqref{label_082}, we can simplify \eqref{eq:ocgm-101} to take for all $k \inn{2}{T-1}$ the form
\begin{equation} \label{eq:ocgm-102}
\begin{gathered}
A_0 (f(x_0) - f(x_k)) \geq B_{k, T} \left( \sum_{i = 1}^{k - 1} a_i^2 \| g_{k} \|_*^2 + a_k^2 \| g_k \|_*^2 \right) \\ + B_{k, T} \left( 2 \sum_{i = 1}^{k - 2} a_{i + 1} \langle s_i, B^{-1} g_{i + 1} \rangle + 2 a_k \langle s_{k - 1}, B^{-1} g_{k} \rangle \right) \overset{\eqref{label_044}}{=} B_{k, T} \| s_k \|_*^2.
\end{gathered}
\end{equation}
\end{proof}
The inequality \eqref{label_120} in Lemma~\ref{label_117} allows us to transform Proposition~\ref{prop:midway-smooth} into a runtime worst-case rate analysis with a simple expression.
\begin{proposition} \label{label_127}
The averaged gradient, defined as the normalized $s_k$, namely
\begin{equation}
\bar{s}_k \defq \frac{1}{A_k - A_0} \sum_{i = 1}^{k} a_i g_i, \quad k \geq 1,
\end{equation}
has a runtime guarantee in squared dual norm given by
\begin{equation} \label{label_044_ogm_rate}
\| \bar{s}_k \|_*^2 < \frac{4 L}{k^2} (f(x_0) - f(x_k)), \quad k \inn{1}{T - 1}.
\end{equation}
\end{proposition}
\begin{proof}
The result in \eqref{label_044_ogm_rate} is directly implied by Proposition~\ref{prop:midway-smooth} and \eqref{label_120} as follows:
\begin{equation}
\begin{gathered}
\| \bar{s}_k \|_*^2 \overset{\eqref{label_044_B_k}}{\leq} \frac{A_0 (f(x_0) - f(x_k))}{(A_k - A_0)^2 B_{k, T}}
\overset{\eqref{label_087}}{=} \frac{L A_0 A_k^2 (f(x_0) - f(x_k))}{(A_k - A_0)^2 A_{T - 1}} \\
\overset{\eqref{label_095}}{=} \frac{2 L \theta_{0, T}^2 (f(x_0) - f(x_k))}{(\theta_{0, T}^2 - 2 \theta_{k, T}^2)^2}
\overset{\eqref{label_120}}{<} \frac{2 L \theta_{0, T}^2 (f(x_0) - f(x_k))}{\left(\theta_{0, T}^2 - \theta_{0, T} \left( \theta_{0, T} - \frac{k}{\sqrt{2}}\right) \right)^2} \\
= \frac{2 L \theta_{0, T}^2}{\theta_{0, T}^2 \frac{k^2}{2}} (f(x_0) - f(x_k)) = \frac{4 L}{k^2} (f(x_0) - f(x_k)), \quad k \inn{1}{T - 1}.
\end{gathered}
\end{equation}
\end{proof}
Looking at both \eqref{label_044_ogm_rate} and \eqref{label_125} we can now understand how Algorithm~\ref{label_099} and the equivalent OGM-G operate at runtime. When $k = 1$ we have the descent rule given by \eqref{label_029}, a result equivalent to Proposition~\ref{prop:midway-smooth} but twice stronger than Proposition~\ref{label_127}. For all subsequent $k \inn{2}{T - 1}$, Proposition~\ref{label_127} reveals that the method actually minimizes the norm of the averaged gradient $\bar{s}_k$. During the last iteration, Algorithm~\ref{label_099} shifts to providing the same rate for $g_T$, given by \eqref{label_125}. The weights $a_k$ increase sharply at we approach the last iteration, explaining the instant one iteration transition from a weighted average of gradients to the last gradient.

\section{Efficiently minimizing the gradient mapping norm}
We now turn our attention to applying the lessons learned in the smooth unconstrained case to the broader problem of composite minimization.

\subsection{Dealing with adaptivity}
In the context of composite problems, the bounds allow for adaptivity.
Specifically, if we obtain a Lipschitz estimate $L_{k + 1}$ by a line-search procedure that ensures the descent rule
\begin{equation} \label{label_129}
f(x_{k + 1}) \leq f(y_{k + 1}) + \langle g_{k + 1}, x_{k + 1} - y_{k + 1} \rangle + \frac{L_{k + 1}}{2} \| x_{k + 1} - y_{k + 1} \|^2, \quad k \inn{0}{T - 1}.
\end{equation}
then we can instantiate Proposition~\ref{label_007} at $x = x_k$ and $x = x_T$, respectively to yield
\begin{align}
F(x_k) &\geq F(x_{k + 1}) + \frac{1}{2 L_{k + 1}} \| g_{k + 1} \|_*^2 + \langle g_{k + 1}, x_k - y_{k + 1} \rangle, \quad k \inn{0}{T-1}, \label{label_130} \\
F(x_T) &\geq F(x_{k + 1}) + \frac{1}{2 L_{k + 1}} \| g_{k + 1} \|_*^2 + \langle g_{k + 1}, x_T - y_{k + 1} \rangle, \quad k \inn{0}{T-2}. \label{label_131}
\end{align}
Note that the range of \eqref{label_131} excludes the redundant $k = T - 1$.

Multiplying every \eqref{label_130} with $A_k$ for $k \inn{0}{T-1}$ and adding them together gives
\begin{equation} \label{label_132}
A_0 F(x_0) + \sum_{k = 1}^{T - 1} a_k F(x_k) \geq A_{T - 1} F(x_T) + \sum_{k = 0}^{T - 1} \frac{A_k}{2 L_{k + 1}} \| g_{k + 1} \|_*^2 + G^{(1)}_T.
\end{equation}
Likewise, multiplying every \eqref{label_131} with $a_{k + 1}$ for $k \inn{0}{T-2}$ and taking the summation results in
\begin{equation} \label{label_133}
(A_{T - 1} - A_0) F(x_T) \geq \sum_{k = 1}^{T - 1} a_k F(x_k) + \sum_{k = 0}^{T - 2} \frac{a_{k + 1}}{2 L_{k + 1}} \| g_{k + 1} \|_*^2 + G^{(2)}_T.
\end{equation}
Adding together \eqref{label_132} with \eqref{label_133}, canceling the matching terms, we obtain
\begin{equation} \label{label_134}
A_0 (F(x_0) - F(x_T) \geq \sum_{k = 0}^{T - 2} \frac{A_{k + 1}}{2 L_{k + 1}} \| g_{k + 1} \|_*^2 + \frac{A_{T - 1}}{2 L_{T}} \| g_{T} \|_*^2 + G^{(1)}_T + G^{(2)}_T.
\end{equation}
Expanding $G^{(1)}_T + G^{(2)}_T$ using \eqref{label_072} finally produces
\begin{equation} \label{eq:ocgm-g-main}
\begin{gathered}
A_0 (F(x_0) - F(x_T)) \geq \sum_{k = 1}^{T - 1} \left( \frac{A_k}{2 L_k} - \frac{a_k}{L_k} - a_k^2 B_{k, T} \right) \| g_{k} \|_*^2 + \frac{A_{T - 1}}{2 L_T} \| g_T \|_*^2 \\
+ \sum_{k = 1}^{T - 2} \left( A_{k} b_{k} - \frac{1}{L_{k + 1}} - 2 a_{k + 1} B_{k + 1, T} \right) \langle s_k, B^{-1} g_{k + 1} \rangle \\+ \left( A_{T - 1} b_{T - 1} - \frac{1}{L_{T}} \right) \langle s_{T - 1}, B^{-1} g_{T} \rangle.
\end{gathered}
\end{equation}
The inequality \eqref{eq:ocgm-g-main} contains all the information we need to construct an adaptive method with convergence guarantees. We list a sufficient condition that is a direct consequence of \eqref{eq:ocgm-g-main} in the sequel.
\begin{proposition} \label{label_135}
If the square matrix $C \in \mathbb{R}^{T} \times \mathbb{R}^{T}$ is positive semidefinite, where $C$ is defined as
\begin{equation} \label{label_136}
C_{i,j} \defq \left\{\begin{array}{ll}
\frac{A_i}{2 L_i} - \frac{a_i}{L_i} - a_i^2 B_{i, T}, & i = j \leq T - 1, \\
C_{T, T} \geq 0, & i = j = T, \\
a_i \left( A_{j - 1} b_{j - 1} - \frac{1}{L_j} - 2 a_j B_{j, T} \right), & i < j \leq T - 1, \\
a_i \left( A_{T - 1} b_{T - 1} - \frac{1}{L_{T}} \right), & i < j = T, \\
0, & i > j,
\end{array} \right.
\end{equation}
then we obtain a guarantee on the last iterate composite gradient norm, given by
\begin{equation} \label{label_137}
\| g_T \|_*^2 \leq \frac{A_0}{\frac{A_{T - 1}}{2 L_T} - C_{T,T}} ( F(x_0) - F(x_T) ).
\end{equation}
\end{proposition}
\begin{proof}
The inequality \eqref{eq:ocgm-g-main} can be refactored using \eqref{label_136} as
\begin{equation} \label{label_138}
A_0 (F(x_0) - F(x_T)) \geq \sum_{i = 1}^{T} \sum_{j = 1}^{T} C_{i, j} \langle g_i, B^{-1} g_j \rangle + \left(\frac{A_{T - 1}}{2 L_T} - C_{T,T} \right) \| g_T \|_*^2.
\end{equation}
We define $Q \in \mathbb{R}^T \times \mathbb{R}^T$ with $Q_{i,j} \defq \langle g_i, B^{-1} g_j \rangle$, $i,j \inn{1}{T}$. The desired result in \eqref{label_137} follows readily from \eqref{label_138} and the fact that $Q$ is positive semidefinite whereby
\begin{equation}
\sum_{i = 1}^{T} \sum_{j = 1}^{T} C_{i, j} \langle g_i, B^{-1} g_j \rangle = \Tr(C Q) \geq 0.
\end{equation}
\end{proof}
Proposition~\ref{label_135} allows for limited adaptivity in the Lipschitz estimates. The implementation of line-search while retaining an optimal $\mathcal{O}(1/k^2)$ rate on $\|g_T\|_*^2$ in \eqref{label_137} remains a very interesting open problem that we leave for future research.

\subsection{An optimized method}
The most aggressive convergence guarantee update occurs when we force $C = 0$ in Proposition~\ref{label_135}. We thus have
\begin{equation} \label{label_139}
\| g_T \|_*^2 \leq \frac{2 A_0 L_T}{A_{T - 1}} ( F(x_0) - F(x_T) ) \leq \frac{2 A_0 L_T}{A_{T - 1}} ( F(x_0) - F(x^*) ).
\end{equation}

However, the presence of the reverse accumulator terms in \eqref{label_136} entails that, during every iteration $k \inn{0}{T - 2}$, the algorithm must have access to the output of the line-search procedure of \emph{all subsequent} iterations. Such knowledge violates our oracle model and hence we restrict ourselves to the scenario in which $L_k = L_0$, $k \inn{1}{T}$, with $L_0$ being an initial guess of the Lipschitz constant. The estimate $L_0$ must satisfy the line-search condition at \emph{every} iteration. If $L_0$ is found invalid at any one iteration $k \inn{0}{T - 1}$, then we need to start over and multiplicatively increase it. Eventually, the value of $L_0$ will equal or surpass $L$, at which stage the line-search condition will always be satisfied. In the sequel we address the situation when the line-search condition passes at every iteration. We thus have
\begin{gather}
A_k b_k = \frac{1}{L_0} + 2 a_{k + 1} B_{k + 1, T}, \quad k \inn{1}{T-2}, \label{label_140}\\
A_{T - 1} b_{T - 1} = \frac{1}{L_0}, \label{label_141} \\
a_k (1 + L_0 a_{k} B_{k, T}) = \frac{A_{k}}{2}, \quad k \inn{1}{T-1}. \label{label_142}
\end{gather}
We extend the weight sequences in Algorithm~\ref{label_021} by imposing the values
\begin{gather}
a_T = A_{T - 1}, \quad A_T = A_{T - 1} + a_T = 2 A_{T - 1}, \label{label_143} \\
b_0 = \frac{1}{L_0 a_1}, \quad \bar{c}_1 = \frac{A_1}{2 L_0 a_1}. \label{label_144}
\end{gather}
The above notation allows us to formulate simple unified expressions for the weight sequences $a_k$ and $b_k$ as well as their accumulators in the following result.
\begin{proposition} \label{label_145}
Equalities \eqref{label_140}, \eqref{label_141} and \eqref{label_142} along with their extensions in \eqref{label_143} and \eqref{label_144} imply
\begin{gather}
a_k = \frac{a_{k + 1}}{A_{k + 1}} \left( \sqrt{a_{k + 1}^2 + A_k A_{k + 1}} - a_{k + 1} \right), \quad k \inn{1}{T - 1}, \label{label_146} \\
b_k = \frac{1}{L_0 a_{k + 1}}, \quad k \inn{0}{T - 1}, \label{label_147}\\
B_{k, T} = \frac{A_k - 2 a_k}{2 L_0 a_k^2}, \quad k \inn{1}{T}, \label{label_148} \\
c_k = \bar{c}_{k} = \frac{A_k}{2 L_0 a_{k}}, \quad k \inn{1}{T - 1}. \label{label_149}
\end{gather}
\end{proposition}
\begin{proof}
Rearranging \eqref{label_142} produces \eqref{label_148} for $k \inn{1}{T - 1}$. For $k = T$ we likewise have $B_{T, T} = 0 \overset{\eqref{label_143}}{=} \frac{A_T - 2 a_T}{2 L_0 a_T^2}$.

Our assumption in \eqref{label_144} gives \eqref{label_147} for $k = 0$. Also, \eqref{label_141} and \eqref{label_143} imply \eqref{label_147} for $k = T - 1$. Substituting \eqref{label_148} in \eqref{label_140} implies for all $k \inn{1}{T-2}$ that
\begin{equation}
A_k b_k \overset{\eqref{label_148}}{=} \frac{1}{L_0} + 2 a_{k + 1} \frac{A_{k + 1} - 2 a_{k + 1}}{2 L_0 a_{k + 1}^2} \overset{\eqref{label_027}}{=} \frac{1}{L_0} + \frac{A_{k} - a_{k + 1}}{L_0 a_{k + 1}} = \frac{A_k}{L_0 a_{k + 1}}.
\end{equation}
We have thus proven \eqref{label_147} for the entire range $k \inn{0}{T - 1}$.

Expanding \eqref{label_147} and \eqref{label_148} in the reverse accumulator definition \eqref{label_028} produces an expression only in the weights $a_k$ and their accumulators for all $k \inn{1}{T - 1}$
\begin{equation} \label{label_150}
\frac{A_k - 2 a_k}{2 L_0 a_k^2} \overset{\eqref{label_148}}{=} B_{k, T} \overset{\eqref{label_028}}{=} B_{k + 1, T} - b_k \overset{\eqref{label_147}}{=} \frac{A_{k + 1} - 2 a_{k + 1}}{2 L_0 a_{k + 1}^2} - \frac{1}{L_0 a_{k + 1}}.
\end{equation}
Rearranging \eqref{label_150} gives the following second-order equation in $a_k$:
\begin{equation} \label{label_151}
\frac{A_{k + 1}}{2} a_k^2 + a_{k + 1}^2 a_k - \frac{a_{k + 1}^2 A_k}{2} = 0, \quad k \inn{1}{T - 1}.
\end{equation}
The equation \eqref{label_151} has only one positive solution given by \eqref{label_146}.

We can use \eqref{label_144} to extend the range of \eqref{label_057} to $k = 0$. Refactoring \eqref{label_140} using the extended \eqref{label_057} and \eqref{label_046} gives
\begin{equation}
\begin{gathered}
\bar{c}_{k + 1}\overset{\eqref{label_057}}{=} A_k b_k - a_{k + 1} B_{k + 1, T} \overset{\eqref{label_140}}{=} \frac{1}{L_0} + a_{k + 1} B_{k + 1, T} \overset{\eqref{label_046}}{=} c_{k + 1} \\ \overset{\eqref{label_148}}{=} \frac{1}{L_0} + a_{k + 1} \frac{A_{k + 1} - 2 a_{k + 1}}{2 L_0 a_{k + 1}^2} = \frac{A_k}{2 L_0 a_{k}}, \quad k \inn{0}{T - 2}.
\end{gathered}
\end{equation}
\end{proof}

Proposition~\ref{label_145} shows how the pairs $(a_k, A_k)$, $k \inn{1}{T-1}$ can be computed in reverse, by starting with $(a_T, A_T) = (A_{T - 1}, 2 A_{T - 1})$, recursing backwards using $A_k = A_{k + 1} - a_{k + 1}$ and \eqref{label_146}. We finally obtain $A_0$ from $A_0 = A_1 - a_1$.

We have now collected all the building blocks needed to construct a method. Substituting the weights from Proposition~\ref{label_145} into the template in Algorithm~\ref{label_021} allows us to formulate our Optimized Composite Gradient Method for minimizing the composite Gradient mapping norm (OCGM-G) listed in Algorithm~\ref{label_152}.

\begin{algorithm}[h]
\caption{An Optimized Composite Gradient Method for minimizing the composite Gradient mapping norm (OCGM-G)}
\label{label_152}
\begin{algorithmic}[1]
\STATE \textbf{Input:} $x_0 \in \mathbb{E}$, $L_0 > 0$, $T \geq 2$, optionally $A_{T - 1} > 0$ or default $A_{T - 1} = 1$
\STATE \textbf{Output:} $y_T$, $g_T$, $x_T$
\STATE $A_{T} = 2 A_{T - 1}$ \label{label_152_a_start}
\STATE $a_T = A_{T - 1}$
\STATE Precompute the pairs $(a_k, A_k)$, $k \inn{1}{T-1}$ and $A_0$ in reverse using \newline{} $A_k = A_{k + 1} - a_{k + 1}$, $a_k = \frac{a_{k + 1}}{A_{k + 1}} \left( \sqrt{a_{k + 1}^2 + A_k A_{k + 1}} - a_{k + 1} \right)$, $A_0 = A_1 - a_1$ \label{label_152_a_end}
\STATE $s_0 = 0$ \label{label_152_s0}
\FOR{$k = 0, \ldots{}, T - 1$}
\STATE $y_{k + 1} = \left\{\begin{array}{ll}x_0, & k = 0 \\ x_k - \frac{1}{L_0 a_{k + 1}} B^{-1} s_k, & k \geq 1 \end{array}\right.$ \label{label_152_y}
\STATE $x_{k + 1} = T_{L_0}(y_{k + 1})$
\STATE $g_{k + 1} = L_{0} B (y_{k + 1} - x_{k + 1})$ \label{label_152_L_new}
\IF {$f({x}_{k + 1}) > f(y_{k + 1}) + \langle f'(y_{k + 1}), {x}_{k + 1} - y_{k + 1} \rangle + \frac{L_{0}}{2} \| {x}_{k + 1} - y_{k + 1} \|^2$} \label{label_152_Ls_start}
\STATE $y_T = y_{k + 1}$, $g_T = g_{k + 1}$, $x_T = x_{k + 1}$
\STATE break and report line-search failure
\ENDIF \label{label_160_short}

\IF {$k = T - 1$}
\STATE break and report line-search success
\ENDIF \label{label_160}
\STATE $s_{k + 1} = s_k + a_{k + 1} g_{k + 1}$ \label{label_161}
\ENDFOR
\end{algorithmic}
\end{algorithm}

\subsection{Equivalent forms}
The closed form of the step sizes $b_k$ in Proposition~\ref{label_145} allows us to write Algorithm~\ref{label_152} in a very simple extrapolated form. For $y_2$ we have
\begin{equation} \label{label_162}
y_2 = x_1 - b_1 a_1 g_1 \overset{\eqref{label_147}}{=} x_1 - \frac{1}{L_0 a_2} a_1 L_0(x_0 - x_1) = x_1 + \frac{a_1}{a_2} (x_1 - x_0).
\end{equation}
The extrapolation and correction terms for subsequent oracle points are given for all $k \inn{2}{T - 1}$ by
\begin{equation} \label{label_163}
\mathcal{E}_k \overset{\eqref{label_033}}{=} \frac{b_k}{b_{k - 1}} \overset{\eqref{label_147}}{=} \frac{a_k}{a_{k + 1}}, \quad \mathcal{C}_k \overset{\eqref{label_033}}{=} b_k \left(L_0 a_k - \frac{1}{b_{k - 1}}\right) \overset{\eqref{label_147}}{=} 0.
\end{equation}
The rule in \eqref{label_162} is merely an extension of \eqref{label_163} for $k = 1$ so we can write the extrapolated form of OCGM-G very simply as shown in Algorithm~\ref{label_164}. All weights $a_k$, $k \inn{1}{T}$, are proportional to $A_{T - 1}$ and we can set this latter value to $1$.

\begin{algorithm}[h]
\caption{Extrapolation based OCGM-G}
\label{label_164}
\begin{algorithmic}[1]
\STATE \textbf{Input:} $x_0 \in \mathbb{E}$, $T \geq 2$
\STATE \textbf{Output:} $y_T$, $g_T$, $x_T$
\STATE Precompute $(a_k, A_k)$, $k \inn{1}{T}$, and $A_0$ using lines~\ref{label_152_a_start}-\ref{label_152_a_end} of Algorithm~\ref{label_152}
\FOR{$k = 0, \ldots{}, T - 1$}
\STATE $y_{k + 1} = \left\{\begin{array}{ll}x_0, &k = 0 \\
x_k + \frac{a_k}{a_{k + 1}}(x_k - x_{k - 1}), & k \geq 1 \end{array}\right.$
\STATE $x_{k + 1} = T_{L_{0}}(y_{k + 1})$
\STATE $g_{k + 1} = L_{0} B (y_{k + 1} - x_{k + 1})$
\STATE Check the termination criteria listed in lines~\ref{label_152_Ls_start}-\ref{label_160} of Algorithm~\ref{label_152}
\ENDFOR
\end{algorithmic}
\end{algorithm}

To write down an auxiliary sequence form of OCGM-G, we first notice that the second order equation in \eqref{label_151} can be refactored as
\begin{equation} \label{label_165}
A_{k + 1} a_k^2 = a_{k + 1}^2 (A_k - 2 a_k), \quad k \inn{1}{T - 1}.
\end{equation}
The multiplier in the oracle point update rule is then given for all $k \inn{1}{T - 2}$ by
\begin{equation} \label{label_166}
d_k \overset{\eqref{label_046}}{=} \frac{b_k}{B_{k, T}} \overset{\eqref{label_147}} = \frac{\frac{1}{L_0 a_{k + 1}}}{\frac{A_k - 2 a_k}{2 L_0 a_k^2}} = \frac{2 a_k^2}{a_{k + 1} (A_k - 2 a_k)} \overset{\eqref{label_165}}{=} \frac{2 a_{k + 1}}{A_{k + 1}}.
\end{equation}
Putting together \eqref{label_166} and \eqref{label_149}, extending \eqref{label_166} to $k = 0$, gives the concise formulation listed in Algorithm~\ref{label_167}.

\begin{algorithm}[h]
\caption{Auxiliary sequence form of OCGM-G}
\label{label_167}
\begin{algorithmic}[1]
\STATE \textbf{Input:} $x_0 \in \mathbb{E}$, $T \geq 2$
\STATE \textbf{Output:} $y_T = v_{T - 1}$
\STATE Precompute $(a_k, A_k)$, $k \inn{1}{T}$, and $A_0$ using lines~\ref{label_152_a_start}-\ref{label_152_a_end} of Algorithm~\ref{label_152}
\STATE $v_0 = x_0$
\FOR{$k = 0, \ldots{}, T - 2$}
\STATE $y_{k + 1} = \left(1 - \frac{2 a_{k + 1}}{A_{k + 1}} \right) x_k + \frac{2 a_{k + 1}}{A_{k + 1}} v_k$
\STATE $x_{k + 1} = T_{L_{k + 1}}(y_{k + 1})$
\STATE $g_{k + 1} = L_{k + 1} B (y_{k + 1} - x_{k + 1})$
\STATE Check the termination criteria listed in lines~\ref{label_152_Ls_start}-\ref{label_160_short} of Algorithm~\ref{label_152}
\STATE $v_{k + 1} = v_k - \frac{A_{k + 1}}{2 L_0 a_{k + 1}} B^{-1} g_{k + 1}$
\ENDFOR
\end{algorithmic}
\end{algorithm}

Likewise, substituting \eqref{label_149} and \eqref{label_147} into Algorithm~\ref{label_021_diak} gives a two auxiliary sequence formulation shown in Algorithm~\ref{label_168}.

\begin{algorithm}[h]
\caption{Two auxiliary sequence form of OCGM-G}
\label{label_168}
\begin{algorithmic}[1]
\STATE \textbf{Input:} $x_0 \in \mathbb{E}$, $T \geq 2$
\STATE \textbf{Output:} $y_T = v_{T - 1}$
\STATE Precompute $(a_k, A_k)$, $k \inn{1}{T}$, and $A_0$ using lines~\ref{label_152_a_start}-\ref{label_152_a_end} of Algorithm~\ref{label_152}
\STATE $v_0 = x_0$
\STATE $s_0 = 0$
\FOR{$k = 0, \ldots{}, T - 2$}
\STATE $y_{k + 1} = \frac{1}{A_{k + 1}} \left( A_k x_k + a_{k + 1} v_k \right) - \frac{1}{2 L_0 a_{k + 1}} s_k$
\STATE $x_{k + 1} = T_{L_{k + 1}}(y_{k + 1})$
\STATE $g_{k + 1} = L_{k + 1} B (y_{k + 1} - x_{k + 1})$
\STATE Check the termination criteria listed in lines~\ref{label_152_Ls_start}-\ref{label_160_short} of Algorithm~\ref{label_152}
\STATE $s_{k + 1} = s_k + a_{k + 1} g_{k + 1}$
\STATE $v_{k + 1} = v_k - \frac{A_{k + 1}}{2 L_0 a_{k + 1}} B^{-1} g_{k + 1}$
\ENDFOR
\end{algorithmic}
\end{algorithm}

\subsection{Relationship to FISTA-G} We have seen in the previous section that the extrapolation property of the auxiliary points is a \emph{consequence} of our template. Interestingly, allowing for adaptivity under the constraints in Proposition~\ref{label_135} preserves this feature. On the other hand, FISTA-G was designed using extrapolation as an \emph{assumption} raising the question as the whether the two approaches produce similar methods. We provide the answer in the following result.

\begin{proposition} \label{prop:ocgm_fista-g}
Supplying $L_0 = L$ to Algorithm~\ref{label_152} and removing the line-search condition checks reduces OCGM-G to the formulation of FISTA-G in \cite{ref_014}.
\end{proposition}
\begin{proof}
We note that \eqref{label_141} extends the range of $k$ in \eqref{label_140} to $k \inn{1}{T - 1}$.
Substituting $A_k$ out of the extended \eqref{label_140} using \eqref{label_142} yields
\begin{equation} \label{label_169}
2 a_k (1 + L a_k B_{k, T}) b_k = \frac{1}{L} + 2 a_{k + 1} B_{k + 1, T}, \quad k \inn{1}{T - 1}.
\end{equation}
Refactoring \eqref{label_147} gives $a_k = \frac{1}{L b_{k - 1}}$ for all $k \inn{1}{T}$ that applied to \eqref{label_169} produces
\begin{equation} \label{label_170}
\frac{2}{b_{k - 1}} \left(1 + \frac{B_{k, T}}{b_{k - 1}} \right) b_k = 1 + 2 \frac{B_{k + 1, T}}{b_k}, \quad k \inn{1}{T - 1}.
\end{equation}
We retain the range $k \inn{1}{T - 1}$ throughout the remainder of the proof. Multiplying both sides of \eqref{label_170} by $b_{k - 1}^2 b_k$ gives
\begin{equation} \label{label_170_norm}
2 b_k^2 (b_{k - 1} + B_{k, T}) = b_{k - 1}^2 (b_k + 2 B_{k + 1, T}).
\end{equation}
Further substituting out the step sizes $b_k$ in \eqref{label_170_norm} using \eqref{label_028} results in a quadratic recursion for their accumulators only, namely
\begin{equation} \label{label_172}
2 (B_{k, T} - B_{k + 1, T})^2 B_{k - 1} = (B_{k - 1, T} - B_{k, T})^2 (B_{k, T} + B_{k + 1, T}).
\end{equation}
The equation \eqref{label_172} with unknown $B_{k - 1, T}$ can be written in canonical form as
\begin{equation} \label{label_173}
(B_{k, T} + B_{k + 1, T}) B_{k - 1, T}^2 - 2 \mathcal{Q}^{(1)}_{k + 1, T} B_{k - 1, T} + B_{k, T}^2 (B_{k, T} + B_{k + 1, T}) = 0,
\end{equation}
where
\begin{equation} \label{label_174}
\begin{aligned}
\mathcal{Q}^{(1)}_{k + 1, T} &\defq B_{k, T} (B_{k, T} + B_{k + 1, T}) + (B_{k, T} - B_{k + 1, T})^2 \\
&= 2 B_{k, T}^2 - B_{k, T} B_{k + 1, T} + B_{k + 1, T}^2.
\end{aligned}
\end{equation}
The quadratic equation in \eqref{label_173} admits only one solution strictly greater than $B_{k, T}$, given by
\begin{equation} \label{label_175}
B_{k - 1, T} = \frac{1}{B_{k, T} + B_{k + 1, T}} \left(\mathcal{Q}^{(1)}_{k + 1, T} + \sqrt{ \mathcal{Q}^{(2)}_{k + 1, T} } \right),
\end{equation}
where
\begin{equation} \label{label_176}
\begin{gathered}
\mathcal{Q}^{(2)}_{k + 1, T} \defq \left(\mathcal{Q}^{(1)}_{k + 1, T}\right)^2 - B_{k, T}^2 (B_{k, T} + B_{k + 1, T})^2 = (B_{k, T} - B_{k + 1, T})^4 \\
+ 2 (B_{k, T} - B_{k + 1, T})^2 B_{k, T} (B_{k, T} + B_{k + 1, T}) = (B_{k, T} - B_{k + 1, T})^2 \left(3 B_{k, T}^2 + B_{k + 1, T}^2 \right).
\end{gathered}
\end{equation}
Expanding \eqref{label_174} and \eqref{label_176} into \eqref{label_175} yields the update found in FISTA-G using the notation $(\varphi_{\on{FISTA-G}})_{k + 1} = B_{k, T}$, namely
\begin{equation} \label{label_177}
B_{k - 1, T} = \frac{2 B_{k, T}^2 - B_{k, T} B_{k + 1, T} + B_{k + 1, T}^2 + (B_{k, T} - B_{k + 1, T}) \sqrt{ 3 B_{k, T}^2 + B_{k + 1, T}^2 } }{B_{k, T} + B_{k + 1, T}}.
\end{equation}
Combining \eqref{label_162} and \eqref{label_147} allows us to extend the extrapolation factor in \eqref{label_033} to $k = 1$. Rewriting this extended expression in terms of the reverse accumulators using \eqref{label_028}, also extended to $k = 0$, gives the extrapolation rule
\begin{equation} \label{label_178}
y_{k + 1} = x_k + \frac{B_{k, T} - B_{k + 1,T}}{B_{k - 1, T} - B_{k,T}} (x_k - x_{k - 1}).
\end{equation}
Taking the updates in \eqref{label_178} and \eqref{label_177} with the notation $(K_{\on{FISTA-G}}) = T - 1$, $(x_{\on{FISTA-G}})_k = y_k$, $(x^{\oplus}_{\on{FISTA-G}}) = x_k$ and $(\varphi_{\on{FISTA-G}})_{k + 1} = B_{k, T}$ we recover FISTA-G in extrapolated form as described in \cite{ref_014}.
\end{proof}

\subsection{Convergence analysis} \label{label_179}

We have seen in Proposition~\ref{label_135} and in \eqref{label_139} that the convergence rate on the composite gradient norm is governed by the sequence $(A_k)_{k \inn{0}{T}}$. Its behavior can be more easily analyzed using the sequence $(t_k)_{k \inn{1}{T}}$, which we define as
\begin{equation} \label{label_180}
t_k \defq \frac{A_k}{a_k}, \quad k \inn{1}{T}.
\end{equation}
A related sequence, under the notation $(c_{\on{FISTA-G}})_k = t_{T - k - 1} - 1$, $k \inn{0}{T - 2}$, has been independently studied in \cite{ref_014}. However, as we will see in Section~\ref{label_223}, our notation establishes a clear link with the $\theta_{k, T}$ sequence in Section~\ref{label_116} as well as the $t_k$ sequence used in the FISTA~\cite{ref_002} variant of FGM, also satisfying \eqref{label_180}. We first study how $A_k/A_0$ can be lower bounded using this new object.
\begin{lemma} \label{label_181}
The sequence $(t_k)_{k \inn{1}{T}}$ satisfies
\begin{align}
t_k (t_k - 2) &= t_{k + 1} (t_{k + 1} - 1), & k &\inn{1}{T - 1} \label{label_182} \\
t_k &\geq t_{k + l} + \frac{l}{2}, &k &\inn{1}{T - 1}, \quad l \inn{1}{T - k} \label{label_183} \\
\frac{t_{k + 1}(t_{k + 1} - 1)}{t_{k}(t_{k} - 1)} &\geq \frac{t_{k + 2} - 1}{t_{k + 2}}, & k & \inn{1}{T - 2}, \label{label_184} \\
\frac{A_k}{A_0} &\geq \frac{t_1^2}{t_{k - 1}(t_{k - 1} - 1)}, & k & \inn{3}{T}. \label{label_185}
\end{align}
\end{lemma}
\begin{proof}
The forward accumulator expression \eqref{label_027} extends via \eqref{label_143} to $k = T - 1$ and we have
\begin{equation} \label{label_186}
\frac{t_{k + 1} - 1}{t_k} \overset{\eqref{label_180}}{=} \frac{\frac{A_{k + 1}}{a_{k + 1}} - 1}{\frac{A_k}{a_k}} \overset{\eqref{label_027}}{=} \frac{\frac{A_{k}}{a_{k + 1}}}{\frac{A_k}{a_k}} = \frac{a_k}{a_{k + 1}}, \quad k \inn{1}{T - 1}.
\end{equation}
Dividing both sides of \eqref{label_165} by $a_k a_{k + 1}^2$ we have
\begin{equation} \label{label_187}
\frac{A_{k + 1}}{a_{k + 1}} \frac{a_k}{a_{k + 1}} = \frac{A_k}{a_k} - 2, \quad k \inn{1}{T - 1}.
\end{equation}
Substituting the definition \eqref{label_180} as well as \eqref{label_186} into \eqref{label_187} yields
\begin{equation} \label{label_182_raw}
t_{k + 1} \frac{t_{k + 1} - 1}{t_k} = t_k - 2, \quad k \inn{1}{T - 1}.
\end{equation}
Multiplying both sides of \eqref{label_182_raw} with $t_k$ gives \eqref{label_182}.

Let us consider \eqref{label_182} to be a quadratic equation with unknown $t_k$. From the definition in \eqref{label_180} it follows that $t_k > 1$, $k \inn{1}{T}$. Hence, the only viable solution of \eqref{label_182} is given for all $k \inn{1}{T - 1}$ by
\begin{equation} \label{label_189}
t_k = 1 + \sqrt{t_{k + 1}^2 - t_{k + 1} + 1} = 1 + \sqrt{\left(t_{k + 1} - \frac{1}{2}\right)^2 + \frac{3}{4}} \geq t_{k + 1} + \frac{1}{2}.
\end{equation}
Iterating \eqref{label_189} we obtain \eqref{label_183}.

Consequently, we have for all $k \inn{1}{T - 2}$ that
\begin{equation}
\frac{t_{k + 2} - 1}{t_{k + 2}} = 1 - \frac{1}{t_{k + 2}} \overset{\eqref{label_183}}{\leq} 1 - \frac{1}{t_k - 1} = \frac{t_k - 2}{t_k - 1} \overset{\eqref{label_182}}{=} \frac{t_{k + 1} (t_{k + 1} - 1)}{t_k (t_k - 1)}.
\end{equation}
The convergence guarantee ratio satisfies for all $k \inn{3}{T}$
\begin{equation}
\begin{gathered}
\frac{A_k}{A_0} = \prod_{i = 1}^{k} \frac{A_i}{A_{i - 1}} = \prod_{i = 1}^{k} \frac{t_i}{t_i - 1}
= \frac{t_1}{t_1 - 1} \frac{t_2}{t_2 - 1} \prod_{i = 1}^{k - 2} \frac{t_{i + 2}}{t_{i + 2} - 1}
\overset{\eqref{label_184}}{\geq} \frac{t_1}{t_1 - 1} \frac{t_2}{t_2 - 1} \\ \prod_{i = 1}^{k - 2} \frac{t_i (t_i - 1)}{t_{i + 1} (t_{i + 1} - 1)}
= \frac{t_1}{t_1 - 1} \frac{t_2}{t_2 - 1} \frac{t_1 (t_1 - 1)}{t_{k - 1}(t_{k - 1} - 1)} \geq \frac{t_1^2}{t_{k - 1}(t_{k - 1} - 1)}.
\end{gathered}
\end{equation}
\end{proof}

\subsubsection{Last iterate guarantee}
Lemma~\ref{label_181} can be used to provide an accurate bound for the gradient mapping norm of the last iterate.
\begin{proposition} \label{label_190}
The composite gradient norm of the last iterate has a worst-case rate of
\begin{equation}
\| g_T \|_*^2 \leq \frac{\mathcal{G}_l L_0}{(T + \mathcal{T}_l)^2} ( F(x_0) - F(x_T) ), \quad l \inn{1}{T - 2},
\end{equation}
where the constants $\mathcal{G}_l$ and $\mathcal{T}_l$ do not depend on $T$ and are, respectively, given by
\begin{equation}
\mathcal{G}_l \defq 8 t_{T - l} (t_{T - l} - 1) \frac{A_{T - l + 1}}{A_{T - 1}}, \quad \mathcal{T}_l \defq 2 t_{T - l} - (l + 1).
\end{equation}
\end{proposition}
\begin{proof}
We can upper bound the final convergence guarantee ratio as
\begin{equation} \label{label_191}
\begin{gathered}
\frac{A_0}{A_{T - 1}} = \frac{A_0}{A_{T - l + 1}} \frac{A_{T - l + 1}}{A_{T - 1}} \overset{\eqref{label_185}}{\leq} \frac{t_{T - l} (t_{T - l} - 1)}{t_1^2} \frac{A_{T - l + 1}}{A_{T - 1}} \\
\overset{\eqref{label_183}}{\leq} \frac{t_{T - l} (t_{T - l} - 1)}{(t_{T - l} + \frac{T - l -1}{2})^2} \frac{A_{T - l + 1}}{A_{T - 1}}
= 4 \frac{t_{T - l} (t_{T - l} - 1)}{(T + \mathcal{T}_l)^2} \frac{A_{T - l + 1}}{A_{T - 1}}, \quad l \inn{1}{T - 2}.
\end{gathered}
\end{equation}
From \eqref{label_139} we have for all $l \inn{1}{T - 2}$ that
\begin{equation}
\| g_T \|_*^2 \leq \frac{2 A_0 L_0}{A_{T - 1}} ( F(x_0) - F(x_T) ) \overset{\eqref{label_191}}{\leq} \frac{\mathcal{G}_l L_0}{(T + \mathcal{T}_l)^2} ( F(x_0) - F(x_T) ).
\end{equation}
\end{proof}
By increasing the value of $l$, we obtain via Proposition~\ref{label_190} progressively more accurate estimates of the rate, as listed in Table~\ref{label_192}.
\begin{table}[h]
\caption{Convergence guarantee coefficients $\mathcal{G}_l $ and $\mathcal{T}_l$ for exponentially increasing $l$}
\label{label_192}
\setlength{\tabcolsep}{4.5pt}
\centering \footnotesize
\begin{tabular}{rrr} \toprule
\multicolumn{1}{c}{$l$} & \multicolumn{1}{c}{$\mathcal{G}_l$} & \multicolumn{1}{c}{$\mathcal{T}_l$} \\ \midrule
1 & 75.7128129 & 3.4641016 \\
2 & 65.0097678 & 3.7883403 \\
5 & 59.1019986 & 4.4316284 \\
10 & 57.5220421 & 5.0803315 \\
100 & 56.6821551 & 7.9500002 \\
1000 & 56.6675000 & 11.2936222 \\
10000 & 56.6673352 & 14.7315296 \\
100000 & 56.6673335 & 18.1833371 \\ \bottomrule
\end{tabular}
\end{table}

However, the rates in Table~\ref{label_192} are only valid when $T \geq l + 2$. For completeness, we also provide a single unified rate for all $T \geq 1$, given by
\begin{equation} \label{label_193}
\| g_T \|_*^2 \leq \frac{56.67 L_0}{(T + 4)^2} ( F(x_0) - F(x_T) ) \leq \frac{56.67 L_0}{(T + 4)^2} ( F(x_0) - F(x^*) ).
\end{equation}
Proposition~\ref{label_190} implies \eqref{label_193} for $T \geq 1000$. It can be verified numerically that \eqref{label_139} gives \eqref{label_193} over of the range $T \inn{2}{999}$. Finally, \eqref{label_029} strictly yields \eqref{label_193} for $T = 1$.

If we choose to run ACGM starting at $x_0$ for $T$ iterations, followed by Algorithm~\ref{label_152} for $T$ iterations, we obtain a convergence guarantee with respect to the initial distance to optimum as
\begin{equation}
\| g_T \|_* \overset{\eqref{label_193}}{\leq} \frac{10.65 L_u}{(T + 1)(T + 4)} d(x_0),
\end{equation}
where $L_u = \max\{\gamma_d L_0, \gamma_u L\}$ (see Section~\ref{label_010}) is the worst-case estimate of $L$ encountered by ACGM, which is assumed to pass the line-search condition for Algorithm~\ref{label_152} at every iteration.

\subsubsection{Averaged gradient mapping rate}

Seeing how for OGM-G minimizes an averaged gradient for all $k \inn{1}{T - 1}$ with a rate of $\mathcal{O}(1/k^2)$ before providing a last iterate with the desired guarantee, the question arises whether OCGM-G exhibits the same behavior. To this end, we first state a counterpart to Proposition~\ref{prop:midway-smooth}.
\begin{proposition} \label{prop:midway-comp}
The state variables of OCGM-G satisfy
\begin{equation} \label{eq:midway-comp}
\frac{a_{k}}{2 L_0} \| g_k \|_*^2 + B_{k, T} \| s_{k} \|_*^2 \leq A_0 (F(x_0) - F(x_k)), \quad k \inn{1}{T}.
\end{equation}
\end{proposition}
\begin{proof}
Refactoring \eqref{label_148} produces
\begin{equation} \label{label_194}
\frac{A_{k - 1}}{2 L_0} - a_k^2 B_{k, T} = \frac{a_k}{2 L_0}, \quad k \inn{1}{T}.
\end{equation}

When $k = 1$, combining $s_1 = a_1 g_1$, \eqref{label_194} as well as \eqref{label_029} with $L_1 = L_0$ produces \eqref{eq:midway-comp}. Likewise, when $k = T$, $B_{T,T} = 0$, \eqref{label_143} and \eqref{label_139} imply \eqref{eq:midway-comp}.

Throughout the remainder of this proof we consider the range $k \inn{2}{T - 1}$. We substitute $T$ with $k$ in \eqref{eq:ocgm-g-main} and set $L_i = L_0$ for all $i \inn{1}{k}$. Noting that for all $i \inn{1}{k}$ it holds that $B_{i, k} = B_{i, T} - B_{k, T}$ we obtain
\begin{equation} \label{eq:ocgm-gk}
\begin{gathered}
A_0 (F(x_0) - F(x_k)) \geq \sum_{i = 1}^{k - 1} \left( \frac{A_i}{2 L_0} - \frac{a_i}{L_0} - a_i^2 (B_{i, T} - B_{k, T}) \right) \| g_{i} \|_*^2 \\
+ a_k^2 B_{k, T} \| g_k \|_*^2 + \left(\frac{A_{k - 1}}{2 L_0} - a_k^2 B_{k, T}\right)\| g_k \|_*^2 \\
+ \sum_{i = 1}^{k - 2} \left( A_{i} b_{i} - \frac{1}{L_0} - 2 a_{i + 1} (B_{i + 1, T} - B_{k, T}) \right) \langle s_i, B^{-1} g_{i + 1} \rangle \\
+ \left( A_{k - 1} b_{k - 1} - \frac{1}{L_0} - 2 a_k B_{k, T} + 2 a_k B_{k, T} \right) \langle s_{k - 1}, B^{-1} g_{k} \rangle.
\end{gathered}
\end{equation}
Applying \eqref{label_140} and \eqref{label_142} to \eqref{eq:ocgm-gk} we obtain
\begin{equation} \label{eq:ocgm-gred}
\begin{gathered}
A_0 (F(x_0) - F(x_k)) \geq \left(\frac{A_{k - 1}}{2 L_0} - a_k^2 B_{k, T}\right) \| g_k \|_*^2 \\
+ B_{k, T} \left( \sum_{i = 1}^{k - 1} a_i^2 \| g_{k} \|_*^2 + a_k^2 \| g_k \|_*^2 + 2 \sum_{i = 1}^{k - 2} a_{i + 1} \langle s_i, g_{i + 1} \rangle + 2 a_k \langle s_{k - 1}, g_{k} \rangle \right).
\end{gathered}
\end{equation}
Simplifying the coefficient of the $\| g_k \|_*^2$ term in \eqref{eq:ocgm-gred} using \eqref{label_194} and rearranging terms using \eqref{label_044} completes the proof.
\end{proof}

Proposition~\ref{prop:midway-comp} describes the convergence behavior for all iterations. The current gradient mapping $g_k$ has a rate governed by $a_k$, which increases sharply during the last iterations to reach the final guarantee in \eqref{label_139} and the rate of the weighted sum of the gradients $s_k$ has the same expression as in Proposition~\ref{prop:midway-smooth}.

We next seek to determine a simple rate for the averaged gradient mapping. To compute a lower bound on this rate, we first need several technical results concerning three functions of one variable. The proof can be found in Appendix~\ref{label_224}.
\begin{lemma} \label{label_195}
The scalar functions $q(x)$, $r(x, \kappa)$ and $\bar{r}(x, \kappa)$, respectively defined for all $x > 1$ and $\kappa \geq 2$ with $\kappa \in \mathbb{R}$ as
\begin{equation} \label{label_196}
\begin{gathered}
q(x) \defq \frac{(x - 1)^2}{x}, \quad r(x, \kappa) \defq \frac{1}{\kappa^2}q\left( \frac{\left(\frac{\kappa}{2} + x - 1\right)^2}{x (x - 1)} \right), \\
\bar{r} \defq (x - 1) (x - 3) r(x, \kappa),
\end{gathered}
\end{equation}
satisfy
\begin{gather}
q(x) > 0, \quad q'(x) > 0, \quad x > 1, \label{label_197} \\
r(x, 2) = \lim_{\kappa \rightarrow \infty} r(x, \kappa) = \frac{1}{4 x (x - 1)}, \quad x > 1, \label{label_198} \\
r(x, \kappa) \geq r(x, 2), \quad x > 1, \quad \kappa \geq 2, \label{label_199} \\
\bar{r}(x, \kappa) > \bar{r}(y, \kappa), \quad x > y \geq 3, \quad \kappa \in \{2\} \cup [3, \infty). \label{label_200}
\end{gather}
\end{lemma}

We can now provide the main rate result.
\begin{proposition} \label{prop:avgrad-comp}
The dual norm of the averaged composite gradient mapping satisfies at runtime
\begin{equation} \label{eq:avg-mid-comp}
\frac{a_{k}}{2 L_0 A_0} \| g_k \|_*^2 + \mathcal{S}_k \| \bar{s}_{k} \|_*^2 \leq F(x_0) - F(x_k), \quad k \inn{2}{T - 1},
\end{equation}
with a rate
\begin{equation} \label{label_201}
\mathcal{S}_k \defq B_{k, T} \frac{(A_k - A_0)^2}{A_0} \geq \frac{\mathcal{R}_k}{L_0} k^2, \quad \mathcal{R}_k \defq \frac{t_k (t_k - 2)}{2} r(t_{k - 1}, k),
\end{equation}
where
\begin{equation} \label{label_202}
\mathcal{R}_k > 0.05, \quad k \inn{2}{T - 3}, \quad \mathcal{R}_{T - 2} > 0.048, \quad \mathcal{R}_{T - 1} > 0.03.
\end{equation}
\end{proposition}
\begin{proof}
We first eliminate the reverse accumulator $B_{k, T}$ from the rate expression of the averaged gradient $\bar{s}_k$ by replacing it with the almost linearly decreasing $t_k$ sequence:
\begin{equation}
\mathcal{S}_k \overset{\eqref{label_148}}{=} \frac{A_k - 2 a_k}{2 L_0 a_k^2} \frac{(A_k - A_0)^2}{A_0} \overset{\eqref{label_180}}{=} \frac{1}{2 L_0} t_k (t_k - 2) q\left( \frac{A_k}{A_0} \right), \quad k \inn{2}{T - 1}.
\end{equation}
We have $A_k > A_0$. Combining the increasing property of $q$ in \eqref{label_197} with the lower bound in \eqref{label_185} we get for all $k \inn{2}{T - 1}$
\begin{equation} \label{label_203}
\mathcal{S}_k \geq \frac{1}{2 L_0} t_k (t_k - 2) q\left( \frac{t_1^2}{t_{k - 1}(t_{k - 1} - 1)} \right) \overset{\eqref{label_183}}{\geq} \frac{k^2}{2 L_0} t_k (t_k - 2) r(t_{k - 1}, k).
\end{equation}
Rearranging \eqref{label_203} gives \eqref{label_201}.

The values $t_T = 2$ and $t_{T - 5} \approx 5.21581 > 5$ imply by \eqref{label_183} that
\begin{equation} \label{label_204}
t_k \geq 2, \quad k \inn{1}{T}, \quad t_k > 5, \quad k \inn{1}{T - 5}.
\end{equation}

From \eqref{label_189} we have
\begin{equation} \label{label_205}
t_k - 1 = \sqrt{t_{k + 1}^2 - t_{k + 1} + 1} < t_{k + 1}, \quad k \inn{1}{T - 1}.
\end{equation}
Lower bounding $t_k$ in the definition of $\mathcal{R}_k$ in \eqref{label_201} using \eqref{label_205} we obtain
\begin{equation} \label{label_206}
\mathcal{R}_k \geq \frac{(t_{k - 1} - 1) (t_{k - 1} - 3)}{2} r(t_{k - 1}, k) = \frac{1}{2} \bar{r}(t_{k - 1}, k), \quad k \inn{2}{T - 1}.
\end{equation}
We can use \eqref{label_204} to provide a rough lower bound for $\bar{r}(t_{k - 1}, k)$ as
\begin{equation} \label{label_207}
\bar{r}(t_{k - 1}, k) \overset{\eqref{label_200}}{>} \bar{r}(5, k) \overset{\eqref{label_199}}{\geq} \bar{r}(5, 2) \overset{\eqref{label_198}}{=} 0.1, \quad k \inn{2}{T - 4}.
\end{equation}
Putting together \eqref{label_206} and \eqref{label_207} we obtain \eqref{label_202} for $k \inn{2}{T - 4}$.

We also have $\mathcal{R}_{T - 3} \approx 0.06055 > 0.05$, $\mathcal{R}_{T - 2} \approx 0.04869 > 0.048$, and $\mathcal{R}_{T - 1} \approx 0.03076 > 0.03$ yielding \eqref{label_202} for $k \inn{T-3}{T-1}$.
\end{proof}

\subsection{Relationship to FGM}

The introduction of the $t_k$ sequence in Section~\ref{label_179} has revealed a startling similarity between OCGM-G and the original form of FGM introduced in \cite{ref_018}, later extended to composite objectives as FISTA in \cite{ref_002}. To make the correspondence more clear, we limit ourselves in this section to studying methods that know the number of iterations $T \geq 2$ in advance and do not attempt to estimate any problem parameters, being fully aware of $L$. We denote such schemes as reduced.

In FGM, the sequence $t_k = A_k / a_k$ increases almost linearly whereas in OCGM-G the sequence has the same expression \eqref{label_180} in terms of convergence guarantees and \emph{decreases} almost in the same way, namely
\begin{equation}
(t_{\on{OCGM-G}})_k \geq \frac{T + 4 - k}{2}, \quad (t_{\on{FGM}})_k \geq \frac{k + 1}{2}, \quad k \inn{1}{T}.
\end{equation}
Furthermore, the extrapolation rule in Algorithm~\ref{label_164} can be expressed in terms of the $t_k$ sequence using \eqref{label_186}, used in the proof of Lemma~\ref{label_181}. We can thus write down the reduced variant of Algorithm~\ref{label_164} in a form that closely resembles reduced FGM, here shown in Figure~\ref{label_208}. In the reduced setup, the algorithm parameters for FGM can be precomputed as well. Also for FGM we have $(t_{\on{FGM}})_0 = \sqrt{L A_0}$~\cite{ref_005} but we have chosen $t_0 = A_0 = 0$ to eliminate the initial function residual from the convergence analysis.

\begin{figure}[h]
\begin{minipage}[t]{0.47\linewidth}
\raggedright Reduced OCGM-G / FISTA-G

\begin{algorithmic}[1]
\STATE \textbf{Input:} $x_0 \in \mathbb{E}$, $T \geq 2$
\STATE Precompute $t_k > t_T$, $k \inn{0}{T-1}$ \emph{in reverse} starting at $t_T = 2$ using $t_{k + 1} (t_{k + 1} - 1) = t_k (t_k - 2)$.
\STATE $x_{-1} = x_0$
\FOR{$k = 0, \ldots{}, T - 1$}
\STATE $y_{k + 1} = x_k + \frac{t_{k + 1} - 1}{t_k}(x_k - x_{k - 1})$
\STATE $x_{k + 1} = T_{L}(y_{k + 1})$
\ENDFOR
\STATE \textbf{Output:} $g_T = L (y_T - x_T)$
\end{algorithmic}

\end{minipage}
\begin{minipage}[t]{0.47\linewidth}
\raggedright Reduced FGM / FISTA

\begin{algorithmic}[1]
\STATE \textbf{Input:} $x_0 \in \mathbb{E}$, $T \geq 2$
\STATE Precompute $t_k > t_0$, $k \inn{1}{T}$ \emph{forward} starting at $t_0 = 0$ using $t_{k + 1} (t_{k + 1} - 1) = t_k$.
\STATE $x_{-1} = x_0$
\FOR{$k = 0, \ldots{}, T - 1$}
\STATE $y_{k + 1} = x_k + \frac{t_{k} - 1}{t_{k + 1}}(x_k - x_{k - 1})$
\STATE $x_{k + 1} = T_{L}(y_{k + 1})$
\ENDFOR
\STATE \textbf{Output:} $x_T$
\end{algorithmic}
\end{minipage}
\caption{Side-by-side comparison between the reduced variants of OCGM-G and FGM}
\label{label_208}
\end{figure}

\subsection{A complete approach}

Our OCGM-G formulation in Algorithm~\ref{label_152} lacks several desirable characteristics. Its guarantees are expressed in terms of the initial function residual, the number of iterations $T \geq 2$ has to be specified in advance and if the line-search criterion fails at any point the algorithm does not provide an adequate guarantee for its output gradient mapping. All these shortcomings can be alleviated by running OCGM-G within the meta-scheme listed in Algorithm~\ref{label_209}.

\begin{algorithm}[h]
\caption{Quasi-online parameter-free ACGM + OCGM-G}
\label{label_209}
\begin{algorithmic}[1]
\STATE \textbf{Input:} $x_0 \in \mathbb{E}$, $L_0 > 0$, $\gamma_d \leq 1$, $\gamma_u > 1$
\STATE $r_0 = x_0$
\STATE $\Lmax = \bar{L}_0 = L_0$
\STATE $T_0 = 2$
\FOR {$j = 0, 1, ...$}
\STATE $(\bar{r}_{j + 1,0}, \bar{L}_{j + 1}) = \mbox{ACGM}(r_j, \bar{L}_j, \gamma_d, \gamma_u, T_j)$ \label{label_209_acgm}
\STATE $\Lmax := \max\{\Lmax, \mbox{maximal } L_{k + 1} \mbox{ within ACGM}\}$ \label{label_209_lmax}
\FOR {$i = 0, 1, ...$}
\STATE $(\bar{r}_{j + 1, i + 1}, \bar{g}_{j + 1, i + 1}) = \mbox{OCGM-G}(\bar{r}_{j + 1, i}, \Lmax, T_j)$ \label{label_209_ocgm}
\IF {OCGM-G line-search success}
\STATE $r_{j + 1} = \bar{r}_{j + 1, i + 1}$ \label{label_209_r}
\STATE $\bar{g}_{j + 1} = \bar{g}_{j + 1, i + 1}$ \label{label_209_g}
\STATE $n_{j + 1} = i + 1$
\STATE break
\ELSE
\STATE $\Lmax = \gamma_u \Lmax$
\ENDIF
\ENDFOR
\STATE $T_{j + 1} = 2 T_j$
\ENDFOR
\end{algorithmic}
\end{algorithm}

The meta-algorithm runs in cycles, each cycle calling one instance of ACGM followed by one or more instances of OCGM-G. Starting at $r_0 = x_0$, Algorithm~\ref{label_209} runs ACGM for $T_0 = 2$ iterations with an initial Lipschitz estimate of $L_0$. The output $\bar{r}_{1, 0}$ as well as the largest Lipschitz estimate $\Lmax$ so far encountered are supplied to OCGM-G, which is also run with $T = T_0$. If OCGM-G reports a line-search failure, then it is re-run starting at the last point outputted by its previous instance, multiplicatively increasing $\Lmax$ until the line-search criterion passes at every internal iteration. When OCGM-G terminates successfully, the first cycle $j = 0$ is completed. The number of iterations $T$ is doubled and the combination ACGM with looped OCGM-G is called again in a new cycle $j = 1$ and so on. Throughout all cycles, the value of $\Lmax$ increases either when ACGM encounters a higher Lipschitz estimate, or when the line-search criterion of OCGM-G fails. This helps keep the overall number of unfinished instances of OCGM-G to a minimum. No effort is wasted because the output of every ACGM or OCGM-G instance is supplied as input to the next. In this respect our meta-algorithm is a restarting schedule.

Let us review the variables used within Algorithm~\ref{label_209}. Our meta-scheme calls ACGM and OCGM at every cycle $j \geq 0$ with several input and output parameters. ACGM is supplied with $r_j$ and $\bar{L}_j$, which become the starting point $x_0$ and initial Lipschitz estimate $L_0$ internally. The line-search parameters $\gamma_d$ and $\gamma_u$ are passed as-is whereas $T_j = 2^{j + 1}$ denotes the predetermined number of iterations of the ACGM instance. Upon termination, the ACGM state parameters $x_{T_j}$ and $L_{T_j}$ are outputted as $\bar{r}_{j + 1,0}$ and $\bar{L}_{j + 1}$, respectively, with $\bar{L}_{j + 1}$ to be taken up by the next instance of ACGM, which will be called in the next cycle. The variable $\Lmax$ is global and updated by ACGM internally at every iteration $k$ as $\Lmax := \max\{\Lmax, L_{k + 1}\}$ ensuring that no Lipschitz estimate of ACGM exceeds it.

OCGM-G is looped $n_{j + 1}$ times, each call indexed by $i \inn{0}{n_{j + 1} - 1}$ is supplied with $\bar{r}_{j + 1, i}$, $\Lmax$ and $T_j$ becoming $x_0$, $L_0$ and $T$, respectively, within the OCGM-G instance. Every output consists of $\bar{r}_{j + 1, i + 1}$ and $\bar{g}_{j + 1, i + 1}$, which represent $x_T$ and $g_T$ in Algorithm~\ref{label_152}, even when terminated early. Upon line-search success, the cycle ends and outputs the valid point $r_{j + 1} = \bar{r}_{j + 1, n_{j + 1}}$ and gradient mapping $\bar{g}_{j + 1} = \bar{g}_{j + 1, n_{j + 1}}$. The whole meta scheme terminates when either a predetermined computational budget is expended, the composite gradient norm decreases below a predetermined threshold or on any other verifiable criterion is triggered.

Our scheme is parameter-free and, unlike the approach in \cite{ref_013}, it is quasi-online and the optional termination parameters do not influence its operation.

\subsubsection{Convergence analysis}

Next, we study the guarantees provided by our meta-scheme. The called ACGM instances always make progress in function residual with respect to their \emph{input} state whereas Proposition~\ref{prop:midway-comp} ensures the same kind of behavior for OCGM-G, even when terminated early. We thus have
\begin{equation} \label{label_215}
\begin{gathered}
F(r_{j + 1}) = F(\bar{r}_{j + 1, n_{j + 1}}) \leq... F(\bar{r}_{j + 1, i + 1}) \leq F(\bar{r}_{j + 1, i}) \leq ... \leq F(\bar{r}_{j + 1, 0}) \\
\leq F(r_j) \leq ... \leq F(x_0), \quad i \inn{0}{n_{j + 1} - 1}, \quad j \geq 0.
\end{gathered}
\end{equation}
At the end of each cycle $j \geq 0$, we have
\begin{equation} \label{label_216}
\begin{gathered}
\| \bar{g}_{j + 1} \|_*^2 \overset{\eqref{label_193}}{\leq} \frac{56.67 L_u}{T_j^2}\left( F\left(\bar{r}_{j + 1, n_{j + 1} - 1}\right) - F^* \right) \\
\overset{\eqref{label_215}}{\leq} \frac{56.67 L_u}{T_j^2}\left( F(\bar{r}_{j + 1}, 0) - F^* \right)
\leq \frac{113.34 L_u^2}{T_j^4} d(r_j)^2, \quad j \geq 0,
\end{gathered}
\end{equation}
where $T_j = 2^{j + 1}$, $j \geq 0$, and $L_u = \max\{\gamma_d L_0, \gamma_u L\}$, defined in Section~\ref{label_010}, is the worst-case value of $L_{\on{max}}$.

Assuming no line-search failures are reported by OCGM-G, the number of ACGM and OCGM-G iterations needed to produce $\bar{g}_{j + 1}$ is $N = 2 \sum_{i = 1}^{j} T_j = 4 T_j - 4$. Let the total number of OCGM-G line-search failures be designated by $b$. Because $L_{\on{max}}$ is multiplicatively updated after every OCGM-G failure and no failures occur once $L_{\on{max}} \geq L$, it follows that $b$ is upper bounded as $b \leq \left\lceil \frac{log(L) - log(L_0)}{log(\gamma_u)} \right\rceil$.
The upper bound on $N$ is reached when ACGM fails to increase $L_{\max}$ at every instance and all line-search failures of OCGM-G take place during the last cycle, every time during the last OCGM-G iteration. We thus have $N \leq (4 + b) T_j - 4$ implying $T_j \geq N / (4 + b)$ which, combined with \eqref{label_216}, gives
\begin{equation} \label{label_217}
\| g_N \|_* \leq \frac{10.65 (4 + b)^2 L_u}{N^2} d(r_j) \leq \frac{10.65 \left(4 + \left\lceil \frac{log(L) - log(L_0)}{log(\gamma_u)} \right\rceil\right)^2 L_u}{N^2} d(r_j),
\end{equation}
where $g_N$ is the composite gradient mapping obtained after $N$ iterations of ACGM and OCGM-G, $r_j$ is the input of the last cycle $j \geq 0$ and $N = (4 + b) 2^{j + 1} - 4$.

When dealing with elongated level sets, it is possible to express the worst-case bound in \eqref{label_217} using $d(r_0)$ by discarding the progress made during OCGM-G instances while resuming ACGM instances instead of restarting them. Specifically, we can replace line~\ref{label_209_acgm} in Algorithm~\ref{label_209} with
\begin{equation}
(r_{j + 1}, \bar{L}_{j + 1}) = \mbox{resume-ACGM}(r_j, r_0, \bar{L}_j, \gamma_d, \gamma_u, T_{j - 1}),
\end{equation}
where $r_j$ is the point from where ACGM resumes its operation, originally started from $r_0$, with all other internal parameters unchanged from the previous ACGM instance. We also must replace line~\ref{label_209_ocgm} in Algorithm~\ref{label_209} with
\begin{equation}
(\bar{r}_{j + 1, i + 1}, \bar{g}_{j + 1, i + 1}) = \mbox{OCGM-G}(r_{j + 1}, \Lmax, T_j),
\end{equation}
as well as \emph{remove} line~\ref{label_209_r} to discard the primal progress made by OCGM-G instances. The resuming of ACGM halves the number of iterations needed to reach a desired convergence guarantee because this guarantee is centered on $r_0$ and not $r_j$. We can interpret this resuming approach as having ACGM run continuously from $r_0$ with OCGM-G instances branching off whenever the ACGM internal iteration counter reaches $T_j$, with each OCGM instance also being run for $T_j$ iterations. Once an OCGM-G instance successfully terminates and outputs the composite gradient mapping, its branch is discarded. Resuming ACGM lowers the bound on $N$ to $\sum_{i = 1}^{j} T_j + T_0 + \sum_{i = 0}^{j - 1} T_j = 3 T_j - 2$ and worst-case result on $g_N$ becomes
\begin{equation} \label{label_218}
\| g_N \|_* \leq \frac{10.65 (3 + b)^2 L_u}{N^2} d(r_0), \quad N = (3 + b) 2^{j + 1} - 2, \quad j \geq 0.
\end{equation}
This approach has better guarantees on problems where the level sets become elongated as we proceed and $d(r_j)$ increases. However, it is wasteful on well-behaved problems because it discards the progress in function residual made by all but the last instance of OCGM-G. Even assuming OCGM-G does not fail, the waste amounts to at least around half the total computational budget.

\section{Simulations} \label{label_219}
With the guarantees stated, we next verify the effectiveness of our complete approach with preliminary simulations. We have tested our quasi-online parameter-free scheme in Algorithm~\ref{label_209} on instances of two popular composite optimization problems found in machine learning and inverse problems: least absolute shrinkage and selection operator (LASSO)~\cite{ref_025} and non-negative least squares (NNLS). The oracle functions for the two problems are listed in Table~\ref{label_220}.

\begin{table}[h]
\label{label_220}
\caption{Oracle functions of the two tested problems}
\begin{center}
\begin{tabular}{lllll} \toprule
& $f({x})$ & $\Psi({x})$ & $f'({x})$ & $\prox_{\tau}({x})$ \\ \midrule
LASSO & $\frac{1}{2}\|{A} {x} - {b} \|_2^2$ & $\lambda_1 \|{x}\|_1$ & ${A}^T({A}{x} - {b})$ & $\mathcal{T}_{\tau \lambda_1} ({x})$ \\[1mm]
NNLS & $\frac{1}{2}\|{A} {x} - {b} \|_2^2$ & $\sigma_{\mathbb{R}_{+}^n}({x})$ & ${A}^T({A}{x} - {b})$ & $({x})_{+}$ \\ \bottomrule
\end{tabular}
\end{center}
\end{table}

The auxiliary functions that appear in Table~\ref{label_220} are the characteristic function of the positive orthant $\sigma_{\mathbb{R}_{+}^n}(x)$, the rectified linear function $(x)_+$ and the shrinkage operator $\mathcal{T}_{\tau}(x)$, respectively defined for all $x \in \mathbb{E}$ and $\tau > 0$ as
\begin{equation}
\begin{gathered}
\sigma_{\mathbb{R}_{+}^n}(x) \defq \left\{ \begin{array}{ll} 0, & x_j \geq 0, \quad j \inn{1}{n} \\ +\infty, & \mbox{otherwise} \end{array} \right., \\
((x)_+)_j \defq \max\{0, x_j\}, \quad j \inn{1}{n}, \\
\mathcal{T}_{\tau}(x)_j \defq \left\{
\begin{array}{ll}
x_j - \tau, & x_j > \tau \\
0, & |x_j| \leq \tau \\
x_j + \tau, & x_j < -\tau
\end{array} \right., \quad j \inn{1}{n}.
\end{gathered}
\end{equation}
Throughout our experiments we have used the standard Euclidean norm and all random variables were sampled independently identically distributed, unless stated otherwise.

For LASSO, matrix $A$ is $m = 500$ by $n = 500$ with entries sampled from $\mathcal{N}(0, 1)$, the standard normal distribution with mean $0$ and variance $1$. Vector $b$ has entries drawn from $\mathcal{N}(0, 9)$, the regularization parameter is $\lambda_1 = 4$ and the starting point $x_0 \in \mathbb{R}^n$ has all entries sampled from $\mathcal{N}(0, 1)$.

For NNLS, matrix ${A}$ is $m = 1000$ by $n = 10000$ and sparse having $10\%$ of entries non-zero, with the locations chosen uniformly random and with values drawn from $\mathcal{N}(0, 1)$. The starting point $x_0$ is sparse and has $10\%$ of entries equal to $4$ at uniformly random locations. Vector $b$ was generated using $b = A x_0 + e$, where $e$ is dense with entries drawn from $\mathcal{N}(0, 1)$.

We have tested the original ACGM in \cite{ref_009}, the Guess and Check Accumulative Regularization approach in \cite{ref_013} (denoted as GCAR) calling instances of ACGM on the appropriately regularized objective, our meta-scheme in Algorithm~\ref{label_209} employing FISTA~\cite{ref_002} and FISTA-G~\cite{ref_014} instead of ACGM and OCGM-G (denoted by FISTA + FISTA-G) as well as our Algorithm~\ref{label_209} calling ACGM and OCGM-G as per original specification (referred to as ACGM + OCGM-G). We left out from our simulations the resuming variant of ACGM + OCGM-G because it performed worse than standalone ACGM in all our experiments.

To produce a fair comparison between the tested approaches, wherein FISTA + FISTA-G is problem parameter dependent, we have supplied the correct value of the Lipschitz constant to all schemes. The schemes that feature adaptivity, either stand-alone or part of a meta-scheme use the line-search parameters $\gamma_d = 0.9$ and $\gamma_u = 2.0$, as employed and argued in \cite{ref_009}. The target gradient mapping norm was set to $\epsilon = 10^{-8} \|g_1\|_2$, where $g_1 = L (x_0 - T_L(x_0))$. The parameter $\epsilon$ was supplied to GCAR because its implementation cannot function without it whereas all other tested schemes only used $\epsilon$ as the termination criterion. The convergence behavior of the tested methods is shown on LASSO in Figure~\ref{label_221} and on NNLS in Figure~\ref{label_222}. For each problem, we have plotted the gradient mapping norm as well as the function residual.

\begin{figure}[h]
\subfigure[Convergence in composite gradient norm]{\includegraphics[width=0.48\linewidth]{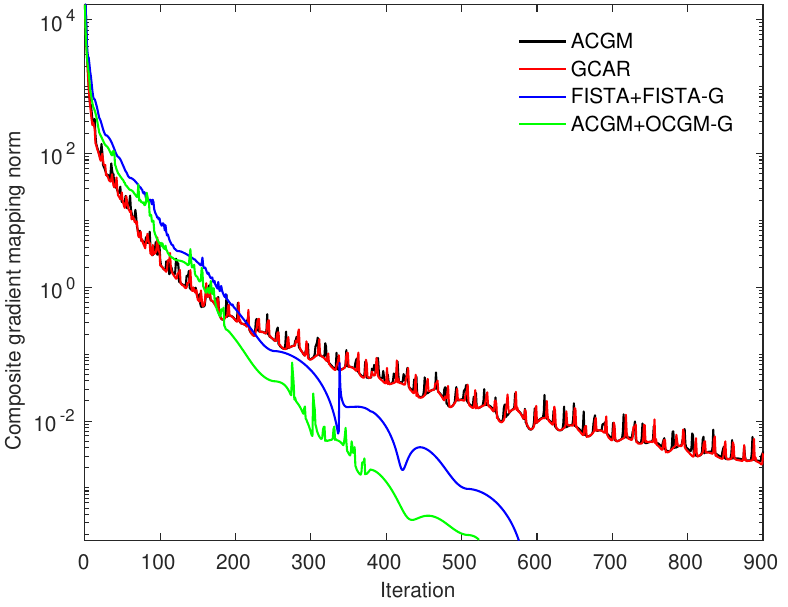}}
\subfigure[Convergence in function residual]{\includegraphics[width=0.48\linewidth]{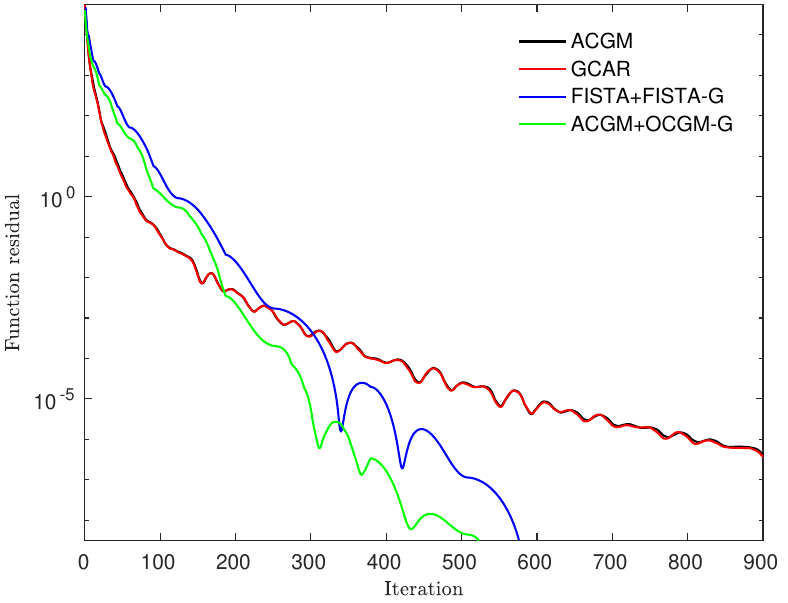}}
\vspace{-0.5em}

\caption{Simulation results on LASSO}
\label{label_221}
\end{figure}

\begin{figure}[h]
\subfigure[Convergence in composite gradient norm]{\includegraphics[width=0.48\linewidth]{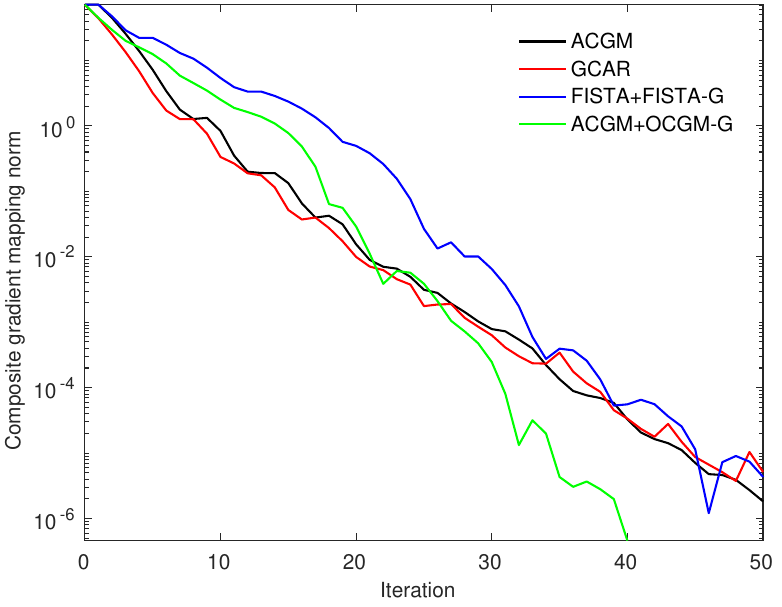}}
\subfigure[Convergence in function residual]{\includegraphics[width=0.48\linewidth]{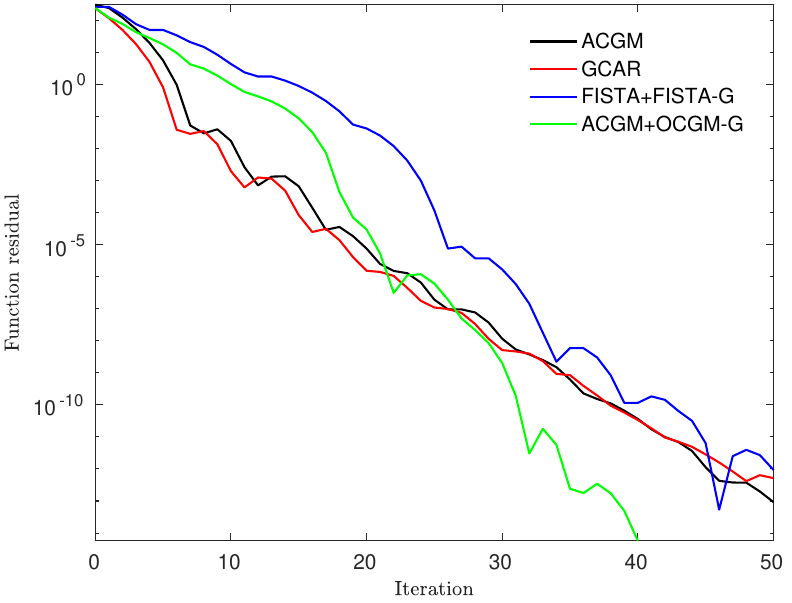}}
\vspace{-0.5em}

\caption{Simulation results on NNLS}
\label{label_222}
\end{figure}

On both problems, the performance of GCAR is almost indistinguishable from ACGM. FISTA + FISTA-G performs worse than ACGM initially on both LASSO and NNLS but manages to overtake ACGM in the higher accuracy regime on LASSO and to match ACGM during the last iterations on NNLS. The adaptivity in our ACGM + OCGM-G ensures that the initial lag with respect to ACGM on both problems is overcome in the high accuracy regime where our approach outperforms all others. The reasons for the poor initial progress made by the approaches based on our meta-scheme is the frequent restarting as well as the fact that OCGM-G and FISTA-G focus on the last iterate. There is no major difference between the convergence behavior for gradient mapping norm and the one for function residual for any tested method on either problem, aside from the function residual convergence begin slightly more stable. The gradient mapping norm jitter is to be expected because ACGM was designed to minimize the function residual whereas the optimized schemes focus on minimizing an averaged gradient mapping most of the time, which allows for large variations in successively outputted gradient mappings.

\section{Discussion} \label{label_223}

In this work we have introduced a simple template for gradient norm minimization given an initial function residual in Algorithm~\ref{label_021}. Unlike Performance Estimation~\cite{ref_004,ref_024} derived approaches, our template is simple and human readable. Simplicity lies in the small number of arithmetic operations performed during each iteration, with division notably lacking. We only use two stand-alone positive weight sequences $(a_k)_{k \inn{1}{T-1}}$ and $(b_k)_{k \inn{1}{T-1}}$. Interestingly, unlike traditional accelerated schemes such as FGM, our template does not employ weight accumulators in its proposed form. Aside from the mandatory oracle point sequence $(y_k)_{k \inn{1}{T}}$, we use the closely related x-point sequence $(x_k)_{k \inn{0}{T}}$ and a memory sequence $(s_k)_{k \inn{0}{T - 1}}$, both objects commonly found in accelerated first-order methods.

The human readability stems not just from the simplicity of the updates, but also from the clear interpretation of all quantities involved. The weights $a_k$ accumulate gradients in $s_k$, and themselves accumulate forward into $A_k$, $k \geq 1$, starting from $A_0 > 0$ to form the final guarantee $A_{T - 1}$ on $\|g_T\|_*^2$ (see \eqref{label_084} and \eqref{label_139}). The direction in which the algorithm progresses is given by $s_k$ with step size $b_k$. The weights $b_k$ also accumulate in reverse, starting with $B_{T, T} = 0$, into the guarantees $B_{k, T}$, $k \inn{1}{T - 1}$, on $\|s_k\|_*^2$ at \emph{runtime}, as given by Propositions \ref{prop:midway-smooth} and \ref{prop:midway-comp}.

The template itself does not employ any inequalities. Even so, without any knowledge of the problem we were able to obtain the sequence $(v_k)_{k \inn{0}{T - 1}}$, initially proposed in \cite{ref_014}, that is updated additively and whose final term matches the last iterate (see \eqref{label_047}). Combining inequalities that describe the problem class under the assumptions of our template is sufficient for the design of optimal schemes for gradient norm minimization. Summing up the inequalities using the weights $a_k$, a technique previously applied to the problem of function residual minimization (see \cite{ref_005} and references therein), followed by the application of the problem-independent properties of our template in Section~\ref{label_062} \emph{directly} yields the optimal scheme for each problem class studied in this work.

We have first shown on smooth unconstrained minimization that applying the inequality describing this problem class \eqref{label_002} under our framework recovers OGM-G~\cite{ref_011}, in the process providing a human readable analysis and pointing out the results missing from the previous attempt in \cite{ref_003}. Our template provides four distinct but equivalent formulations. Instantiating each with the our OGM-G parameters proves the equivalence between the previous formulations of OGM-G in \cite{ref_011}, \cite{ref_014} and \cite{ref_003}, with the necessary corrections in the case of the latter.

Applying the well-known inequality for composite minimization in the proximal gradient setup in \eqref{label_009} produces OCGM-G, an algorithm that allows for adaptivity at runtime. However, the implementation of line-search at runtime while maintaining an optimal rate remains an interesting open problem.
We have therefore limited ourselves to implementing a guess-and-check parameter-free approach in OCGM-G. Removing parameter estimation from OCGM-G recovers FISTA-G in \cite{ref_014}. We provide a means of determining the worst-case rate of OCGM-G with progressively increasing accuracy in Proposition~\ref{label_190} and the highest known global constant on this class in \eqref{label_193}. The four equivalent forms of the template also apply to the composite setup and we provide extrapolated as well as one and two auxiliary sequence reformulations of OCGM-G. Oracle points in OCGM-G can be obtained by extrapolating x-points, a fact derived from our framework, as opposed to FISTA-G whose derivation relies on this being assumed beforehand.

The formulation and convergence analysis of the optimized methods OGM-G and OCGM-G is greatly simplified by the introduction of the sequence $(t_k)_{k \inn{1}{T}}$, defined as $t_k = A_k / a_k$ (see \eqref{label_180}). This is unsurprising, seeing how the first accelerated method ever proposed~\cite{ref_018} was written and analyzed using this sequence exclusively. Our template reveals that the $(\theta_{k,T})_{k \inn{0}{T-1}}$ sequence originally described in \cite{ref_011} is merely a shifted version of $t_k$. Indeed, \eqref{label_095} and \eqref{label_096} imply that $t_k = \theta_{k - 1, T}$, $k \inn{1}{T}$. Whereas in the case of function residual minimization, the sequence increases by at least $1/2$ per term, when minimizing the gradient norm, this sequence \emph{decreases} in the same manner. Writing OCGM-G in extrapolated form using the $t_k$ exclusively produces a simple formulation that closely resembles that of the FGM in \cite{ref_018} and FISTA in \cite{ref_002}, as shown in Figure~\ref{label_208}. The similarity extends to the auxiliary sequence form in Algorithm~\ref{label_167}.

Previous attempts at explaining optimized schemes for minimizing the gradient norm were only able to formulate a guarantee on the last iterate. It was unclear what exactly these methods minimize at runtime. In this work we finally provide an answer by introducing the averaged gradient $\bar{s}_k$, obtained directly from $s_k$ through normalization. We have shown that $\|\bar{s}_k\|^2_*$ converges with a provable $\mathcal{O}(1/k^2)$ rate at every iteration except the last one for both OGM-G in Proposition~\ref{label_127} and OCGM-G in Proposition~\ref{prop:avgrad-comp}. Most of the time, the rate constant higher for the averaged gradient and decreases to that of the last iterate gradient towards the last iterations.

The runtime guarantees ensure that, just like FGM, the optimized methods produce function residuals no worse than at the starting state, even if terminated early. This fact has allowed us to formulate a meta-scheme calling OCGM-G. Algorithm~\ref{label_209} cycles through one instance of ACGM followed by one more instances of OCGM-G, the output of every call being fed as input to the next. When OCGM-G is successful, a new cycle begins, providing our meta-scheme with an optimal rate on the composite gradient norm expressed in terms of the distance to optimum of the cycle starting point.

When the level sets are elongated, it is possible to modify Algorithm~\ref{label_209} to have exponentially longer instances of OCGM-G branch off a main ACGM process at similarly increasing intervals. While desirable on pathological problems, ignoring the primal progress made by OCGM-G may hinder practical performance.

Notwithstanding, the original formulation of Algorithm~\ref{label_209} calling ACGM and OCGM-G (ACGM + OCGM-G) is currently the most advanced methodology for minimizing the composite gradient norm.
It is superior to GCAR because it has a much better proportionality constant and only estimates one parameter (the Lipschitz constant $L$) instead of two ($L$ and the initial distance to optimum $d(r_0)$), despite both approaches spending around half the computational effort minimizing the function residual of the original objective. Moreover, unlike both GCAR and FISTA-G, ACGM + OCGM-G is a truly parameter-free scheme that doesn't rely on a target accuracy or a number of iterations being specified in advance, thus also exhibiting a quasi-online mode of operation.

Preliminary simulation results confirm the superiority of ACGM + OCGM-G, suggesting that our approach is competitive with existing schemes even when the latter are supplied with accurate values of the problem constants.

\section*{Acknowledgements}

This project has received funding from the European Research Council (ERC) under the European Union's Horizon 2020 research and innovation programme (grant agreement No. 788368). We thank prof. Yurii Nesterov for providing insightful suggestions and comments.

\appendix

\section{Proof of Lemma~\ref{label_195}} \label{label_224}

First, \eqref{label_197} directly follows from the definition of $q$, namely
\begin{equation}
q(x) = \frac{(x - 1)^2}{x} > 0, \quad q'(x) = 1 - \frac{1}{x^2} > 0, \quad x > 1.
\end{equation}
Also from the definition of $q$ we have
\begin{equation} \label{label_225}
q(x) = \frac{(x - 1)^2}{x} = x + \frac{1}{x} - 2, \quad x > 1.
\end{equation}
Using \eqref{label_225} we obtain for all $x > 1$
\begin{equation} \label{label_198_k}
\lim_{\kappa \rightarrow \infty} r(x, \kappa) = \lim_{\kappa \rightarrow \infty} \frac{\left(\frac{\kappa}{2} + x - 1\right)^2}{\kappa^2 x (x - 1)} + \lim_{\kappa \rightarrow \infty} \frac{x (x - 1)}{\kappa^2 \left(\frac{\kappa}{2} + x - 1\right)^2} - \lim_{\kappa \rightarrow \infty} \frac{2}{\kappa^2} = \frac{1}{4 x (x - 1)}.
\end{equation}
Likewise,
\begin{equation} \label{label_227}
r(\kappa, 2) = \frac{1}{4} q\left(\frac{x^2}{x (x-1)}\right) \overset{\eqref{label_225}}{=}\frac{1}{4} \left( \frac{x}{x - 1} + \frac{x - 1}{x} - 2 \right) = \frac{1}{4 x (x - 1)}, \quad x > 1.
\end{equation}
Combining \eqref{label_198_k} and \eqref{label_227} gives \eqref{label_198}.

We define the variables $x_1 \defq 2 (x - 1)$ and $x_2 \defq 2 x x_1 = 4 x (x - 1)$ for all $x > 1$. Our assumption $\kappa \geq 2$ is equivalent to $\kappa \geq \frac{x_2}{x_1} - x_1$ which is in turn equivalent to $(\kappa + x_1)^2 - x_2 \geq \kappa (\kappa + x_1)$. The right-hand side is positive and dividing both sides by $\kappa(\kappa + x_1)$ and taking the square we get $\frac{((\kappa + x_1)^2 - x_2)^2}{\kappa^2 (\kappa + x_1)^2} \geq 1$. We thus have
\begin{equation}
r(x, 2) \overset{\eqref{label_198}}{=} \frac{1}{x_2} \leq \frac{\left(\frac{(\kappa + x_1)^2}{x_2} - 1\right)^2}{\kappa^2 \frac{(\kappa + x_1)^2}{x_2}} \overset{\eqref{label_196}}{=} \frac{1}{\kappa^2} q\left( \frac{(\kappa + x_1)^2}{x_2} \right) \overset{\eqref{label_196}}{=} r(x, \kappa), \quad x > 1.
\end{equation}

From now on we consider the range $x \geq 3$ and $\kappa \geq 2$ unless stated otherwise.
We further introduce the variables $\kappa_1 \defq \frac{\kappa}{2} - 1$ and $\kappa_2 \defq \frac{\kappa_1^2}{\kappa - 1}$.
Refactoring the expression for $\bar{r}(x, \kappa)$ in \eqref{label_196} using $\kappa_1$ and $\kappa_2$ yields
\begin{equation} \label{label_228}
\bar{r}(x, \kappa) = (x - 1) (x - 3) \frac{\left(\frac{(x + \kappa_1)^2}{x (x - 1)} - 1 \right)^2}{\frac{(x + \kappa_1)^2}{x (x - 1)}}
= \frac{(\kappa - 1)^2}{\kappa^2} \frac{(x - 3)(x + \kappa_2)^2}{x (x + \kappa_1)^2}.
\end{equation}
The partial derivative w.r.t. $x$ in \eqref{label_228} is given by
\begin{equation} \label{label_229}
\frac{\partial}{\partial x} \bar{r}(x, \kappa) = \frac{(\kappa - 1)^2}{\kappa^2} \frac{x + \kappa_2}{x^2 (x + \kappa_1)^3} \bar{q}(x, \kappa).
\end{equation}
where $\bar{q}(x, \kappa) \defq (2 \kappa_1 - 2 \kappa_2 + 3) x^2 + (9 \kappa_2 - 3 \kappa_1) x + 3 \kappa_1 \kappa_2$. The positivity of $\kappa$ implies $\kappa_1 > \kappa_2$ and hence $2 \kappa_1 - 2 \kappa_2 + 3 > 0$.
The discriminant of the equation $\bar{q}(x, \kappa) = 0$ in $x$ is given by
\begin{equation} \label{label_230}
\mathcal{D}(\kappa) \defq (9 \kappa_2 - 3 \kappa_1)^2 - 12 \kappa_1 \kappa_2 (2 \kappa_1 - 2 \kappa_2 + 3).
\end{equation}
We can substitute $\kappa_2$ out of \eqref{label_230} thereby obtaining
\begin{equation} \label{label_231}
\mathcal{D}(\kappa) = \frac{3 \kappa_1^2}{(2 \kappa_1 + 1)^2} \bar{\mathcal{D}}(\kappa_1), \quad \bar{\mathcal{D}}(\kappa_1) \defq 3 - 18 \kappa_1 - 29 \kappa_1^2 - 8 \kappa_1^3.
\end{equation}
The derivative $\bar{\mathcal{D}}'(\kappa_1) = -18 - 58 \kappa_1 - 24 \kappa_1^2$ is negative for all $\kappa_1 \geq 0$ implying that $\bar{\mathcal{D}}$ is strictly decreasing on this range.

We distinguish two cases. When $\kappa = 2$, \eqref{label_228} implies $\bar{r}(x, \kappa) = \frac{(\kappa - 1)^2}{\kappa^2} \frac{(x - 3)}{x}$, a strictly increasing function in $x$.
When $\kappa \geq 3$, we have $\kappa_1 \geq 0.5$ and $\bar{\mathcal{D}}(0.5) \approx -2.6719 < 0$ and hence the discriminant is negative. The partial derivative in \eqref{label_229} is positive, completing the proof of \eqref{label_200}.
\end{document}